\documentclass{article}

\usepackage{graphicx}
\usepackage{caption}
\usepackage{subcaption}
\usepackage[english]{babel}
\usepackage[utf8]{inputenc}

\usepackage[letterpaper,top=2cm,bottom=2cm,left=3cm,right=3cm,marginparwidth=1.75cm]{geometry}

\usepackage{stmaryrd}
\usepackage{mathtools}
\usepackage{chngcntr}
\usepackage{indentfirst}
\usepackage{amsmath}
\usepackage{graphicx}
\usepackage{amssymb}
\usepackage{amsthm}
\usepackage[colorlinks=true, allcolors=blue]{hyperref}
\usepackage{mathrsfs}
\usepackage{tikz-cd}
\usetikzlibrary{arrows.meta}
\usepackage{tikz}
\usetikzlibrary{arrows.meta,patterns}
\usepackage{stmaryrd}
\usepackage[title]{appendix}

\usepackage[doi = false, isbn = false, url = false, style = alphabetic, maxnames = 100]{biblatex}
    \DeclareFieldFormat{postnote}{#1}
    \AtBeginDocument{%
       \def\MR#1{}
    }
\usepackage{cleveref}
\newtheorem{thm}{Theorem}[subsection]
\newtheorem{lemma}[thm]{Lemma}

\newtheorem{rmk}[thm]{Remark}
\newtheorem{prop}[thm]{Proposition}
\newtheorem{defn}[thm]{Definition}
\newtheorem{conj}[thm]{Conjecture}

\Crefname{thm}{Theorem}{Theorems}
\crefname{lemma}{lemma}{lemmas}
\Crefname{lemma}{Lemma}{Lemmas}
\crefname{cor}{corollary}{corollaries}
\Crefname{cor}{Corollary}{Corollaries}
\crefname{rmk}{remark}{remarks}
\Crefname{rmk}{Remark}{Remarks}
\crefname{prop}{proposition}{propositions}
\Crefname{prop}{Proposition}{Propositions}
\crefname{defn}{definition}{definitions}
\Crefname{defn}{Definition}{Definitions}
\crefname{conj}{conjecture}{conjectures}
\Crefname{conj}{Conjecture}{Conjectures}
\crefname{eg}{example}{examples}
\Crefname{eg}{Example}{Examples}

\usepackage{tikz}
\usetikzlibrary{cd}
\usetikzlibrary{arrows}
\usetikzlibrary{matrix}

\usepackage{bbm}

\newcommand{\Z}{\mathbb Z}

\newcommand{\R}{\mathbb R}

\newcommand{\T}{\mathbb T}
\newcommand{\Q}{\mathbb Q}

\newcommand{\mbL}{\mathbb L}
\newcommand{\C}{\mathbb C}

\newcommand{\HH}{\mathcal H}
\newcommand{\F}{\mathbb F}
\newcommand{\A}{\mathcal A}
\newcommand{\mcL}{\mathcal L}
\newcommand{\mcF}{\mathcal F}
\newcommand{\G}{\mathbb G}
\newcommand{\mf}{\mathfrak}

\newcommand{\M}{\mathcal{M}}
\newcommand{\B}{\mathcal{B}}

\newcommand{\mc}{\mathcal}
\newcommand{\mcC}{\mathcal C}

\newcommand{\DL}{\mathrm{DL}}
\newcommand{\RP}{\mathrm{RP}}
\newcommand{\GL}{\mathrm{GL}}
\newcommand{\SL}{\mathrm{SL}}
\newcommand{\Ad}{\mathrm{Ad}}
\newcommand{\Irr}{\mathrm{Irr}}

\newcommand{\fg}{\mathfrak g}

\newcommand{\ft}{\mathfrak t}

\newcommand{\bt}{\bar{\mathfrak t}}
\newcommand{\bg}{\bar{\mathfrak g}}

\newcommand{\bl}{\bar{\mathfrak l}}

\usepackage{tikz}
\usetikzlibrary{cd}
\usetikzlibrary{arrows}
\usetikzlibrary{matrix}
\tikzset{
    vertex/.style = {
        circle,
        fill=black,
        outer sep=2pt,
        inner sep=1pt
    }}
\usepackage{float}

\title{Stability of certain elements in the Bernstein center}
\author{Sarbartha Bhattacharya and Shenghao Li}
\numberwithin{equation}{subsubsection}
\renewcommand{\theequation}{%
  \ifnum\value{subsection}>0 %
    \ifnum\value{subsubsection}>0 %
      \thesubsubsection.\arabic{equation}
    \else
      \thesubsection.\arabic{equation}
    \fi
  \else
    \thesection.\arabic{equation}
  \fi
}
\numberwithin{thm}{subsubsection}
\renewcommand{\thethm}{%
  \ifnum\value{subsection}>0 %
    \ifnum\value{subsubsection}>0 %
      \thesubsubsection.\arabic{thm}
    \else
      \thesubsection.\arabic{thm}
    \fi
  \else
    \thesection.\arabic{thm}
  \fi
  }

\newtheoremstyle{normalexample}
  {3pt} 
  {3pt} 
  {\normalfont} 
  {} 
  {\bfseries} 
  {.} 
  {.5em} 
  {} 

\theoremstyle{normalexample}

\date{}

\begin{document}
\maketitle

\begin{abstract}
Let $G$ be a split connected reductive algebraic group over a non-archimedean local field $F$ of characteristic zero. The stable center conjecture asserts that the subspace of stable elements in the Bernstein center of $G(F)$ is a sub-algebra, and this subring is conjecturally isomorphic to the regular functions on the moduli of semi-simple Langlands parameters. Bezrukavnikov-Kazhdan-Varshavsky suggested an approach towards the stable center conjecture by explicitly constructing elements and proving their stability directly without assuming LLC, and proved the stability of the depth-$r$ projectors. Following the idea, the first author and Tsao-Hsien Chen produced a family of elements in the center which are conjecturally stable using Deligne-Lusztig parameters attached to irreducible representations. This paper proves the stability of the projectors corresponding to strongly regular Deligne-Lusztig parameters in depth-zero, and for all the positive depth elements with some restrictions on the residue characteristic.

\end{abstract}
\tableofcontents
\section{Introduction}

Let $F$ be a non-archimedean local field of characteristic zero with ring of integers $\mf O$ and residue field $k=\F_q$ of characteristic $p$. We fix a separable closure $\bar{F}$ of $F$. Let $F^{nr}$ denote the maximal unramified extension of $F$ and $F^t$ denote the maximal tamely ramified extension of $F$. Let $G$ be a split connected reductive algebraic group defined over $F$. Fix an $F$-split maximal torus $T\subset G$. Further, we fix a Haar measure $\mu$ on $G(F)$ and let $\HH(G)$ denote the space of compactly supported smooth functions on $G(F)$. Let $R(G)$ denote the category of smooth complex representations of $G(F)$ and $\mf{Z}(G)=\mathrm{End}(\mathrm{Id}_{R(G)})$ denote the Bernstein center. As in \cite[\S 3]{haines2014stable}, the Bernstein center $\mathfrak{Z}(G)$ acts on every irreducible representation of $G(F)$ by a scalar. It can also be described as the algebra of essentially compact invariant distributions on $G(F)$. For further details about the structure and equivalent descriptions of the center, see \cite{bhattacharya2025description,haines2014stable}. 

Denote the set of (isomorphism classes of) smooth irreducible representations of $G(F)$ by $\Irr(G)$. Moy and Prasad in \cite{MP94} attached a non-negative rational number $\rho(\pi)\in \Q_{\geq 0}$ to each $\pi \in \Irr(G)$, called the depth of $\pi$. For any non-negative rational number $r\in\mathbb Q_{\geq 0}$,
let $R(G)_{\leq r}$ (resp, $R(G)_{> r}$) denote the full subcategory of smooth representations whose irreducible subquotients have depth $\leq r $ (resp, depth $>r$). Let $\Irr(G)_r$ denote the set of smooth irreducible representations of $G(F)$ of depth $r$ (similarly $\Irr(G)_{\leq r}$, $\Irr(G)_{<r}$ and $\Irr(G)_{>r}$). The results of Bernstein and Moy-Prasad (\cite{Ber, MP94, moy1996jacquet}) imply that the category $R(G)$ decomposes as a direct sum
 $R(G) = R(G)_{\leq r} \oplus R(G)_{> r}$, and hence the Bernstein center also decomposes as $\mf Z(G) = \mf Z^r(G) \oplus \mf Z^{>r}(G)$. We call $\mf Z^r(G)$ the depth-$r$ Bernstein center, and it acts non-trivially only on irreducible representations in $\Irr(G)_{\leq r}$.

\subsection{The stable center}
An element $z \in \mf Z(G)$ of the center is called stable if its associated invariant distribution $\nu_z$ is stable (for more details, see \cite[\S 3]{bkv2016depthrprojector} and \cite[\S 2.3]{bhattacharya2025description}). This condition is expected to have a spectral interpretation through the local Langlands correspondence. In particular, the stable characters (the appropriate linear combinations of characters of tempered $L$-packets) is conjectured to be stable, and therefore stable distributions should reflect the organization of $L$-packets. Let $\mf Z^{\mathrm{st}}\subset\mathfrak{Z}(G)$ denote the subset of stable elements in the Bernstein center. By definition, it is a vector subspace. The stable center conjecture states that $\mf Z^{\mathrm{st}}(G)\subset \mf Z(G)$ is a subalgebra, and there is a decomposition of the set $\Irr(G)$ of irreducible representations of $G(F)$ into packets via characters of $\mf Z^{\mathrm{st}}(G)$, which is slightly coarser than the conjectural decomposition into $L$-packets as per the Local Langlands correspondence (see \cite[Section 3.1.4]{bkv2013categorical}). \par

The stable Bernstein center has a conjectural definition as regular functions on the moduli space of semisimple Langlands parameters (\cite{scholzeshin2011cohomologycompactunitarygroup,farguesscholze2024geometrizationlocallanglandscorrespondence,bkv2013categorical,chen2023towards}). Under the assumption that Local Langlands correspondence for $G$ and its Levi subgroups is known, and some compatibility of LLC with normalized parabolic induction (see LLC+ as defined in \cite[\S 5.2]{haines2014stable}), Haines in \cite{haines2014stable} gives a description of the stable center as the ring of regular functions on the variety of infinitesimal characters. Under several assumptions including $G$ is quasi-split, LLC+ is known for $G$ and the existence of tempered $L$-packets etc. (see \cite[Proposition 1.1.7]{varma2024some}), Varma showed in \cite{varma2024some} that the different notions of the stable Bernstein center are equivalent, and hence the stable center conjecture is true.  If LLC is known for $G$, the elements of the stable Bernstein center are conjectured to act by the same constant on one infinitesimal class, as defined in \cite[\S 5.2]{haines2014stable}, which are unions of $L$-packets.

Bezrukavnikov-Kazhdan-Varshavsky in \cite{bkv2013categorical} suggested an approach towards the stable center conjecture by constructing enough elements in the center which conjecturally lie in the stable Bernstein center, and proving their stability directly without assuming LLC in cases where it is unknown. As an example of their approach, they gave a construction and proved the stability of the depth-zero projector in \cite{bkv2013categorical} using a geometric approach and homology of affine Springer fibers. Moreover, in \cite{bkv2016depthrprojector}, they gave a description of the depth-$r$ projectors for rational depths as Euler-Poincar\'e sums and proved their stability. The approach for the depth zero projector involves establishing a relationship with the Steinberg character and using the stability of the same. In rational depths, they used Fourier transforms to reduce the question to proving stability of certain subsets in the $p$-adic group, the Lie algebra and the dual Lie algebra. 

In \cite{bhattacharya2024description,bhattacharya2025description}, the first author and Tsao-Hsien Chen gave a description of the depth-$r$ center as an inverse limit of certain Hecke algebras, and constructed elements in the center which act as projectors onto subcategories determined by Deligne-Lusztig parameters. We describe the results and constructions in \cite{bhattacharya2025description} in a bit more detail in the next section. These projector elements were conjectured to be stable (\cite[Conjecture 1.1]{bhattacharya2025description}). In this work, we prove their conjecture for the case of strongly regular parameters in depth zero and uniformly for positive depths, with a long term goal of making progress towards the stable center conjecture via the path suggested by Bezrukavnikov-Kazhdan-Varshavsky.

\subsection{Depth-$r$ center and conjecturally stable elements}

Let $\B(G,F)$ denote the reduced Bruhat-Tits building of $G$. For $x \in \mc B(G,F)$ and $r \in \R_{\geq 0}$, let $G_{x,r}=G(F)_{x,r}$ denote $r$-th Moy-Prasad filtration subgroup of $G(F)$ as defined in \cite{MP94} and $G_{x,r+}:=\cup_{s>r}G_{x,s}$. Let $\A(T,F)=\A(G,T,F)\subset \B(G,F)$ be  the reduced apartment corresponding to $T$, and we fix a chamber $\mcC \subset \A(T,F)$ in the apartment. This is equivalent to fixing a set of standard parahorics containing a fixed Iwahori, since they correspond to the parahoric subgroups of the facets in the closure of the fixed chamber. In \cite{bhattacharya2025description}, they replace the usual simplicial structure on the building $\B(G,F)$ by a refinement by subdividing the facets into smaller parts depending on $m \in \Z_{>0}$ when the depth $r \in \frac{1}{m}\Z_{\geq 0}$. For $r \in  \frac{1}{m}\Z_{>0} $, the vanishing hyperplanes corresponding to the refined affine roots (see \cite[Section 2.5]{moysavin} and \cite[Section 4.1]{bhattacharya2025description}) are used to give $\B(G,F)$ (similarly the apartment, chamber) a refined simplicial structure, and we denote it by $\B_m$( resp. $\mc A_m, \mcC_m$). Let $[\B_m]$ (similarly $[\mc A_m], [\bar{\mcC}_m]$) denote the set of refined facets in $\B_m$( resp. $\mc A_m, \bar{\mcC}_m$). For example, see the case of $\SL_3(\Q_p)$ below. 
 
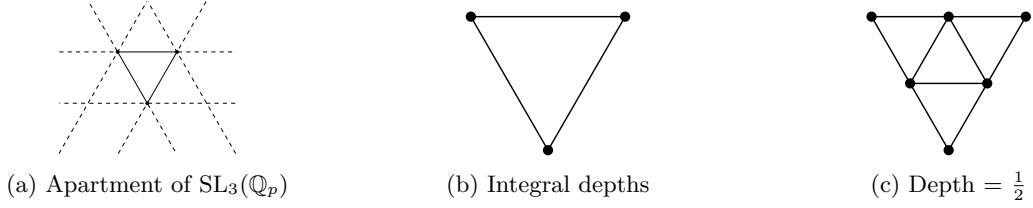
\begin{figure}[H]
    \begin{subfigure}[b]{0.32\textwidth}
        \centering
        \resizebox{0.5\linewidth}{!}{
\begin{tikzpicture}[scale=0.5]
            \node[inner sep=0pt]  (A) at (0,0){};
            \node [inner sep=0pt] (B) at (3,0){};
            \node[inner sep=0pt]  (C) at (1.5, -2.598){};
            \node[inner sep=0pt]  (A1) at (-3,0){};
            \node [inner sep=0pt] (B1) at (6,0){};
            \node[inner sep=0pt]  (C1) at (-1.5, 2.598){};
             \node[inner sep=0pt]  (D) at (3, -5.196){};
             \node[inner sep=0pt]  (D1) at (0, -5.196){};
             \node[inner sep=0pt]  (D2) at (4.5, 2.598){};
             \node [inner sep=0pt] (E) at (6,-2.598){};
              \node [inner sep=0pt] (E1) at (-3,-2.598){};
               \node [inner sep=0pt] (F) at (-3,-5.196){};
              \node [inner sep=0pt] (F1) at (1.5,2.598){};
              \node [inner sep=0pt] (G) at (6,-5.196){};
             
            \draw (A)--(B)--(C)--(A);
            \draw[dashed] (A1)--(A)--(C1);
            \draw [dashed](D1)--(C)--(D);
            \draw [dashed](D2)--(B)--(B1);
             \draw [dashed](E)--(E1);
            \draw [dashed](F)--(F1);
            \draw [dashed](G)--(F1);
            \draw node[vertex](A) at (0,0){};
            \draw node[vertex](B) at (3,0){};
            \draw node[vertex](C) at (1.5, -2.598){};
            \end{tikzpicture}
        }
        \caption{Apartment of $\SL_3(\Q_p)$}
        \label{fig:buildingGL3}
    \end{subfigure}
    \hfill
    \begin{subfigure}[b]{0.32\textwidth}
    \centering
        \resizebox{0.5\linewidth}{!}{
        \begin{tikzpicture}[scale=0.5]
            \node[inner sep=0pt]  (A) at (0,0){};
            \node [inner sep=0pt] (B) at (3,0){};
            \node[inner sep=0pt]  (C) at (1.5, -2.598){};
            \draw (A)--(B)--(C)--(A);
            \draw node[vertex](A) at (0,0){};
            \draw node[vertex](B) at (3,0){};
            \draw node[vertex](C) at (1.5, -2.598){};   
        \end{tikzpicture}
        }
        \caption{Integral depths}   
        \label{fig:integer}
    \end{subfigure}
    \hfill
    \begin{subfigure}[b]{0.32\textwidth}
        \centering
        \resizebox{0.5\linewidth}{!}{
            \begin{tikzpicture}[scale=0.5]
            \node[inner sep=0pt]  (A) at (0,0){};
            \node [inner sep=0pt] (B) at (3,0){};
            \node[inner sep=0pt]  (C) at (1.5, -2.598){};
            \node [inner sep=0pt] (AB) at (1.5,0){};
            \node [inner sep=0pt] (BC) at (2.25,-1.299){};
            \node [inner sep=0pt] (CA) at (0.75,-1.299){};
            \draw (A)--(B)--(C)--(A);
            \draw (AB)--(BC)--(CA)--(AB);
            \draw node[vertex](A) at (0,0){};
            \draw node[vertex](B) at (3,0){};
            \draw node[vertex](C) at (1.5, -2.598){};
            \draw node[vertex](AB) at (1.5,0){};
            \draw node[vertex](CA) at (0.75,-1.299){};
            \draw node[vertex](BC) at (2.25,-1.299){};

        \end{tikzpicture}
        }
        \caption{Depth = $\frac{1}{2}$}
        \label{fig:depthhalf}
    \end{subfigure}

\end{figure}
For $r \in \frac{1}{m}\Z_{\geq 0}$ and $\sigma \in [\B_m]$ a refined facet, note that $G_{\sigma,r}$ and $G_{\sigma,r+}$ are well defined. For each $\sigma \in [\B_m]$, we can define  
$\M_\sigma^r:=C_c^\infty\left(\frac{G(F)/G_{\sigma,r+}}{G_{\sigma,0}}\right)$ to be the algebra of compactly supported smooth functions which are $G_{\sigma,r+}$ bi-invariant and $G_{\sigma,0}$-conjugation invariant. We have a partial order on $[\B_m]$ given by $\sigma' \preceq \sigma $ if $\sigma'$ is contained in the closure of $\sigma$, and this gives a partial order on $[\Bar{\mcC}_m]$. For $\sigma',\:\sigma\in [\Bar{\mcC}_m] $ and $\sigma' \preceq \sigma $, we have a map 
\begin{align*}
\phi^r_{\sigma',\sigma}:\M^r_{\sigma'} &\longrightarrow \M^r_{\sigma}\\
f &\longmapsto f*\delta_{G_{\sigma , r+}}  
\end{align*}
where $\delta_{G_{\sigma , r+}} =\frac{1}{\mu(G_{\sigma,r+})}1_{G_{\sigma,r+}}$ is the unit element in $\mc M_\sigma^r$. Adding some extra morphisms \cite[Section 3]{bhattacharya2025description} to the inverse system over refined facets, the description is proved in \cite{bhattacharya2025description}, generalizing the integral depth results in \cite{bhattacharya2024description}

\begin{thm}[Theorem 3.8,\cite{bhattacharya2025description}]\label{thm:Descriptionrationaldepths}
    There is an explicit algebra isomorphism $[A^r]:  A^r(G) := \lim_{\sigma\in [\Bar{\mcC}_m] } \M^r_{\sigma} \xrightarrow{\simeq} \mf Z^r(G)$ for a split reductive group $G$ and $r \in \frac{1}{m}\Z_{\geq 0}$.
\end{thm}
In \cite{chen2025langlands}, Chen-Debacker-Tsai attached depth-$r$ Deligne-Lusztig parameters (denoted $\DL_r$) to depth-$r$ irreducible representations, and these parameters were shown to be in bijection to certain parameters on the Galois side of local Langlands correspondence called depth -$r$ restricted Langlands parameters (denoted $\RP_r$). We assume $p \nmid |W|$ and hence the depth of a representation lies in $\Z_{(p)}\cap \Q_{\geq 0}$. There is a map (\cite{bhattacharya2025description,chen2025langlands}) $\tilde \Theta_r:\Irr(G)_r \xrightarrow{\bar \Theta_r} \DL_r \xrightarrow{\sim }\RP_r$ for $r \in \Z_{(p)}\cap \Q_{\geq 0}$ such that for $r >0$, the image lies in non-trivial parameters (denoted $\DL^\circ_r,\: \RP_r^\circ$ respectively). Thus, there is a map 
\begin{equation}\label{equation:restricted}
    \bar{\Theta}:\Irr(G) \rightarrow{}\DL^t:=\DL_0\coprod \left(\coprod_{r \in \Z_{(p)}\cap \Q_{> 0}}\mathrm{DL}_r^\circ\right) \xrightarrow{\sim} \RP^t:=\RP_0\coprod \left(\coprod_{r \in \Z_{(p)}\cap \Q_{> 0}}\mathrm{RP}_r^\circ\right)
\end{equation}

 Let $\hat T$ denote the complex dual torus of $T$, $\bt$ the Lie algebra of its reductive quotient and $W$ the absolute Weyl group. As per  \cite[Remark 7.9]{bhattacharya2025description}, $\DL_0$ can be identified with the set $(\hat T//W)^{[q]}$ where the superscript $[q]$ denotes the fixed-point set of the map induced by $t\rightarrow t^q$.
 
 We fix the extension $\mathrm{v}: \bar F^{\times} \rightarrow \Q$ of the normalized valuation on $F$, so that $\mathrm v(F^{\times})=\Z$; set $\mathrm{v}(0)=+\infty$. For $m \in \Z_{>0}$ with $p \nmid m$, fix $\nu_m \in F^t$  such that $\mathrm v(\nu_m)=\frac{1}{m}$. Any element of $\DL_r$ for $r \in \frac{1}{m}\Z$ is the equivalence class of a pair $(\nu,\vartheta)$ with $\mathrm{v}(\nu)=\frac{1}{m}$ and $\vartheta \in \bt^*//W(\bar\F_q)$. Fixing such $\nu_m$ coupled with the fact that $G$ is split identifies $\DL_r$ for $r \in \Z_{(p)}\cap \Q_{> 0}$ with $\bt^*//W(\F_q)$ (see \cite[Section 6.2]{bhattacharya2025description}). Using their description of the center, \cite{bhattacharya2025description} constructs the following maps:
\begin{thm}[\cite{bhattacharya2025description}]\label{thm:projectorconstructions}
    There exist algebra homomorphisms
    \begin{equation*}
        \xi^0:\C[(\hat T//W)^{[q]}]\longrightarrow \mf Z^0(G)
    \end{equation*}
    and for $r \in \Z_{(p)}\cap \Q_{>0}$ 
    \begin{equation*}
         \xi^r: \C[\bt^*//W(\F_q)] \longrightarrow \mf Z^r(G)
    \end{equation*}
    which are compatible with the assignment in \eqref{equation:restricted}, i.e., for $\varphi \in \DL^\circ_r$, $\xi^r(1_\varphi)|_\pi =\mathrm{Id} $ if and only if $\bar{\Theta}_r(\pi)=\varphi$ and otherwise it acts by $0$.  
\end{thm}
The restricted Langlands parameters are conjecturally restrictions of semi-simple $L$-parameters attached to representations (\cite{chen2025langlands}). This motivated the following conjecture in \cite{bhattacharya2025description}
\begin{conj}
      For $r \in \Z_{(p)}\cap\Q_{\geq0}$, we have $\mathrm{Im}(\xi^r)\subset \mf Z^{\mathrm{st}}(G)\cap \mf Z^r(G)$.
\end{conj}
As mentioned in \cite{bhattacharya2025description}, to prove the conjecture, it is enough to prove the stability of the projector elements $\xi^r(1_\varphi)$ for $\varphi\in \DL^t$. In this article, we prove several cases of the conjecture by dividing it into two cases, the depth-zero part and the positive depth part. We prove the stability of the elements corresponding to strongly regular parameters (as defined in Section \ref{sec:Notation}) in depth-zero and all parameters in positive depths under some assumptions on the residue characteristic. We prove the following theorem under some assumptions on the residue characteristic:

\begin{thm}
    If $r=0$ and $[s]\in (\hat T//W)^{[q]}=\mathrm{DL}_0$ is strongly regular, then $\xi^0(1_{[s]})\in \mathfrak{Z}^0(G)$ is stable. If $ r \in \Z_{(p)}\cap\mathbb{Q}_{>0}$, then $\xi^r(1_{\vartheta})\in \mathfrak{Z}^r(G)$ is stable for any $\vartheta\in \bt^*//W(\F_q)=\mathrm{DL}_r$. 
\end{thm}

\begin{rmk}
    For $ r \in \Z_{(p)}\cap\mathbb{Q}_{>0}$, if $\vartheta = \bar 0^r\in \bt^*//W(\F_q)=\mathrm{DL}_r$, the element $\xi^r(1_{\bar 0^r})$ is exactly the projector to $R(G)_{<r}$ as noted in \cite[Proposition 8.7]{bhattacharya2025description} which is known to be stable from \cite{bkv2016depthrprojector}. Our methods work in this case as well.
\end{rmk}

\subsection{The depth-zero argument}

 For any depth-zero irreducible representation $\pi$ of $G(F)$, we choose a Moy--Prasad type $(x,\mathcal{X})$ of $\pi$, where $x\in \mathcal{B}(G,F)$ and $\mathcal{X}$ is an irreducible cuspidal representation of $G_{x,0}/G_{x,0+}$. Let $W_x$ denote the finite Weyl group corresponding to the facet containing $x$. We can map $\mathcal{X}$ to $(\hat{T}//W_x)^{[q]}$, and then map to $(\hat{T}//W)^{[q]}=\mathrm{DL}_0$. Denote the image of $\mathcal{X}$ in $\mathrm{DL}_0$ by $\theta_{x,\mathcal{X}}$. By \cite[Proposition 7.11]{bhattacharya2025description}, $\theta_{x,\mathcal{X}}$ does not depend on the choice of the Moy--Prasad type, and therefore we can simply denote the image by $\bar{\Theta}_0(\pi)$. Then for any $[s]\in \mathrm{DL}_0$, the definition of $\xi^0(1_{[s]})$ is:
$$\xi^0(1_{[s]})(\pi) = \begin{cases} 1, & \text{if depth}(\pi)=0 \text{ and } \bar{\Theta}_0(\pi)=[s], \\ 0, & \text{otherwise }. \end{cases} $$
To prove the depth-zero cases of the theorem, we study the representations selected by the projector:
$$R_{[s]}=\{\pi\in\operatorname{Irr}(G)_0:\bar{\Theta}_0(\pi)=[s]\}.$$
We first consider the elliptic case (as defined in Section \ref{sec:Notation}). Strong regularity and ellipticity force every representation in $R_{[s]}$ to be an extra regular depth-zero supercuspidal representation. Its Moy-Prasad type is obtained from a cuspidal Deligne–Lusztig representation of a finite reductive quotient, associated with an elliptic torus and a character in general position. This allows us to describe the representations using the regular depth-zero construction of \cite{kaletha2019regular}.

The next step is to show that $R_{[s]}$ is a disjoint union of DeBacker-Reeder $L$-packets. The essential compatibility is that stable conjugacy of the torus-character pairs defining these representations preserves their Deligne–Lusztig parameter. Thus, whenever one member of a packet belongs to $R_{[s]}$, the entire packet belongs to $R_{[s]}$.

We then apply the Plancherel formula to the distribution
$$f\longmapsto \bigl(\xi^0(1_{[s]})*f\bigr)(1).$$
Since the selected spectrum is supercuspidal, its Plancherel expansion is expressed in terms of formal degrees and supercuspidal characters, integrated over unitary central characters. The representations in each DeBacker--Reeder packet have the same central character and formal degree. We can therefore group the expansion into scalar multiples of sums of packet characters. These sums are stable by the stability theorem in \cite{debacker2009depth}, so the distribution vanishes on unstable test functions.

For a strongly regular non-elliptic parameter, the selected representations arise by parabolic induction from proper Levi subgroups. We show that every $\pi\in R_{[s]}$ is an irreducible normalized induction of a depth-zero extra regular supercuspidal representation $\sigma$ of a Levi subgroup $L$ whose Deligne--Lusztig parameter is strongly regular elliptic for $L$, and its image in $\mathrm{DL}_0$ is $[s]$.

The proof of stability again uses the Plancherel formula, now organized by Levi subgroups as in \cite{varma2024some}. After regrouping the characters into packets, each Levi contribution vanishes by the elliptic argument. This proves the theorem without an ellipticity assumption on $[s]$.

\subsection{The positive depth argument}
 The proof in positive depths is uniform under certain assumptions on the residue characteristic. In this case, $\DL_r$ is identified with $\bt^*//W(\F_q)$. Let $\fg$ denote the Lie algebra of $G$. The idea is motivated by the construction of dual blobs, for example in \cite{kimspicevarma,kim2004dual}, and the use of Fourier transform to prove stability as in \cite{bkv2016depthrprojector}. Given $\vartheta \in \bt^*//W(\F_q) =\DL_r$ for $r \in \Z_{(p)}\cap \Q_{>0}$, let $E_{r,\vartheta}:=\xi^r(1_{\vartheta})$ denote the projector element. We define a Lie algebra analogue of the projector $\mc E_{r,\vartheta}$ and construct a open, closed and stable subset (the ``dual blob") $\fg^*_{r,\vartheta} \subset \fg^*$ such that the Fourier transform of $1_{\fg^*_{r,\vartheta}}$ is equal to $\mc E_{r,\vartheta}$. A theorem of Waldspurger \cite{waldspurger1997lemme} shows that the Fourier transform preserves stability, hence proving the stability of $\mc E_{r,\vartheta}$.

In order to get the stability of $E_{r,\vartheta}$ on the group, we make use of the existence of quasi-logarithms and their compatibility with Moy-Prasad isomorphism. This coupled with the construction of $E_{r,\vartheta}$ lets us prove its stability whenever such quasi-logarithms and compatibility conditions are known. Our additional restrictions on the group or the residue characteristic come exactly from this.

\subsection{Organization of the article}
We briefly summarise the main goals of each section. In Section \ref{sec:Notation}, we establish the notation and conventions. Section \ref{section:depth-zero} deals with the case of depth-zero elements, starting with the strongly regular elliptic case and finally proving the result for all strongly regular elements. Section \ref{sec:Positivedepths} deals with the case of positive depth elements.

\paragraph{Acknowledgements} The authors thank Charlotte Chan, Tsao-Hsien Chen, Thomas Haines, Tasho Kaletha, Cheng-Chiang Tsai and Sandeep Varma for many useful discussions and answering the multitude of questions that were asked to them at several points. The authors first met and discussed the project during a summer workshop on Character Sheaves on Loop Groups at the University of Minnesota in the summer of 2026, and would like to thank the organizers of the workshop.

\paragraph{Declaration of AI usage} Large language models were only used at the end for proofreading and minor language modifications. The authors would like to declare that there was no use of AI for ideas and results.

\section{Notation and conventions}\label{sec:Notation}

Throughout the article, we assume that $G$ is a split reductive group over a $p$-adic field $F$. Denote the center of $G$ by $Z$ and the absolute Weyl group by $W$. We assume that $p$ does not divide $|W|$. We have some further assumptions on $p$ for the depth-zero case and the positive depth case separately, which will be stated in the corresponding sections.

Let $\mathrm{DL}_0:=(\hat{T}// W)^{[q]}$ denote the space of depth-zero Deligne--Lusztig parameters. We fix an isomorphism
$$\iota:\bar{k}^{\times}\rightarrow (\mathbb{Q}/\mathbb{Z})_{p'},$$
where $(\mathbb{Q}/\mathbb{Z})_{p'}$ denotes the subgroup of elements in $\mathbb{Q}/\mathbb{Z}$ of order prime to $p$. 

For an element $[s]\in \mathrm{DL}_0$ represented by $s\in \hat{T}$, we call it strongly regular elliptic if it satisfies the following conditions:
\begin{itemize}
    \item (strongly regular) The centralizer of $s$ in $\hat{G}$ is $\hat{T}$. In other words, $\{w\in W\mid w(s)=s\}=\{1\}$.
    \item (elliptic) Suppose $s^q=w_{[s]}(s)$. Then we have $(X_*(T_{\mathrm{ad}})\otimes \mathbb{R})^{w_{[s]}}=0$.
\end{itemize}

While working with rational depths, we subdivide the facets in the building $\B(G,F)$ following \cite[Section 2.4]{bhattacharya2025description}, and let $\B_m$ denote the building with refined simplicial structure as mentioned in the introduction. Given $x \in \B(G,F)$, let $\fg_{x,r}:=\fg(F)_{x,r}$ denote the Moy-Prasad filtration subgroup of the Lie algebra, and let $\fg^*_{x,-r}=\fg^*(F)_{x,-r}$ be the same for the dual of the Lie algebra, defined following \cite{MP94}. Similarly, let $\fg_{x,r+}:=\cup_{s>r}\fg_{x,s}$ and $\fg^*_{x,-r+}=\cup_{s<r}\fg^*_{x,-s}$. For $m \in \Z_{>0}$, let $[\B_m]$ denote the set of refined facets in $\B_m$. Note that for $r \in \frac{1}{m}\Z_{\geq 0}$ and $\sigma \in [\B_m]$, $\fg_{\sigma,r}$, $\fg^*_{\sigma ,-r}$ etc. are well defined. For $r>0$ let $\bg_{x,r}:=\frac{\fg_{x,r}}{\fg_{x,r+}}$ and $\bg^*_{x,r}:=\frac{\fg^*_{x,-r}}{\fg^*_{x,-r+}}$. These $\F_q$-vector spaces are dual to each other and $\bg_{\sigma,r}$ and $\bg^*_{\sigma,r}$ can be defined as above in the appropriate setting. 

Let $\G_x$ denote the reductive quotient at $x \in \B(G)$, i.e., $\G_x(\F_q)=G_{x,0}/G_{x,0+}$ and $\T$ denote the reductive quotient of $T$. We denote by $\ft$ the Lie algebra of $ T$ and $\bt=\ft_0/\ft_{0+}$ is the Lie algebra of $\T$. 

\paragraph{}   A subset $\Sigma \subset [\B_m]$ is a subcomplex if $|\Sigma|=\cup_{\sigma \in \Sigma}\sigma \subset \B(G,F)$ is closed. It is convex if $|\Sigma|$ is convex. Let $\Theta_m$ denote the set of convex subcomplexes of $\B_m$.  A facet of dimension $0$ is called a vertex, and we denote the set of vertices in $\B_m$ by $V(\B_m)$. A facet of maximal dimension is called a chamber. 

\section{The depth-zero case}\label{section:depth-zero}

Throughout this section, we further assume that
$$p>(e_F+2)n,\qquad q>|\Phi^+|$$
where $e_F$ is the absolute ramification index of $F/\mathbb{Q}_p$, $n$ is the dimension of a faithful algebraic representation $G\hookrightarrow \GL_n$ and $|\Phi^+|$ denotes the number of positive roots of $G$. The extra assumptions come from \cite[Restriction 12.4.1]{debacker2009depth} and \cite[\S 2.1]{kaletha2019regular}.

\subsection{Strongly regular elliptic case}
In this subsection, we will recall the definition and properties of $\xi^0(1_{[s]})$ when $[s]\in \mathrm{DL}_0$ is strongly regular elliptic, and proves the following theorem:
\begin{thm}\label{3.1}
    Suppose $[s]\in \mathrm{DL}_0$ is strongly regular elliptic. Then $\xi^0(1_{[s]})\in \mathfrak{Z}^0(G)$ is stable.
\end{thm}
\subsubsection{Definition and properties of $\xi^0(1_{[s]})$}

Let $[s]\in \mathrm{DL}_0$ be strongly regular elliptic and let $1_{[s]}$ denote the characteristic function of $[s]$ on $\mathrm{DL}_0$. Following \cite[\S 7]{bhattacharya2025description}, the spectral description of $\xi^0(1_{[s]})\in \mathfrak{Z}^0(G)$ could be understood as follows: let $(\pi,V)$ be an irreducible representation of $G(F)$. If $\pi$ is not of depth-zero, then $\xi^0(1_{[s]})$ acts on $\pi$ by zero. If $\pi$ is of depth-zero, denote its Moy--Prasad type by $(x,\mathcal{X})$, where $x\in \mathcal{B}(G,F)$ is a point in the reduced Bruhat--Tits building, and $\mathcal{X}$ is an irreducible cuspidal representation on the finite quotient $\mathbb{G}_x(k):=G_{x,0}/G_{x,0+}$. Denote the corresponding Weyl group of $\mathbb{G}_x(k)$ by $W_x$. As stated in \cite[\S 7.4]{bhattacharya2025description}, we can attach to $\mathcal{X}$ an element in $(\hat{T}//W_x)^{[q]}$. Let $\theta_{x,\mathcal{X}}$ denote the image of the corresponding element under the map $(\hat{T}//W_x)^{[q]}\rightarrow (\hat{T}//W)^{[q]}=\mathrm{DL}_0$. By \cite[\S 7.4]{bhattacharya2025description}, we have 
$$\xi^0(1_{[s]})|_{\pi}=1_{[s]}(\bar{\Theta}_0(\pi))\mathrm{Id}_V=1_{[s]}(\theta_{x,\mathcal{X}})\mathrm{Id}_V.$$
In other words, $\xi^0(1_{[s]})$ acts on $V$ as the scalar $1_{[s]}(\bar{\Theta}_0(\pi))=1_{[s]}(\theta_{x,\mathcal{X}})$. Define
$$\mathcal{R}_{[s]}:=\{\pi\in \mathrm{Irr}(G)_0\mid \bar{\Theta}_0(\pi)=\theta_{x,\mathcal{X}}=[s]\}.$$
We have the following main proposition for this section:
\begin{prop}\label{regular}
    If $[s]\in \mathrm{DL}_0$ is strongly regular elliptic, then every $\pi\in\mathcal{R}_{[s]}$ is an extra regular depth-zero supercuspidal representation of $G$. More precisely, it contains a Moy--Prasad type at a vertex $x\in \mathcal{B}(G,F)$ of the form
    $$\mathcal{X}=\epsilon_{\mathbb{G}_x(k),\bar{S}} R_{\bar{S}}^{\mathbb{G}_x(k)}(\bar{\theta}),$$
    where $\bar{S}\subset \mathbb{G}_x(k)$ is a maximal elliptic $k$-torus of $\mathbb{G}_x$. In addition, suppose that $S$ is a lift of $\bar{S}$ (see \cite[Lemma 3.4.4(2)]{kaletha2019regular}). Then the stabilizer of $\bar{\theta}$ in $\Omega(S,G)(F)$ is trivial.
\end{prop}
\begin{proof}
    Choose a Moy--Prasad type $(x,\mathcal{X})$ in $\pi$. Without loss of generality, we may assume $x\in \mathcal{A}(G,T,F)$. In particular, $\mathcal{X}$ is irreducible cuspidal. Denote the facet containing $x$ by $\sigma$. Let $W_x$ denote the Weyl group of the finite reductive group $\mathbb{G}_x(k)$. Since $\pi\in \mathcal{R}_{[s]}$, after conjugation, we can assume that $w_{[s]}\in W_x$. Then since $(X_*(T_{\mathrm{ad}})\otimes \mathbb{R})^{w_{[s]}}=0$, we can see that $\sigma$ cannot be positive dimensional. Therefore, $x$ must be a vertex. By \cite[Proposition 6.6]{moy1996jacquet}, $\pi$ must be an irreducible supercuspidal representation. We denote the Deligne--Lusztig parameter of $\mathcal{X}$ in $(\hat{T}//W_x)^{[q]}$ by $[s]_x$, which is a preimage of $[s]$.

    By \cite[Corollary 7.7]{deligne1976representations}, there exist $\bar{S}\subset \mathbb{G}_x(k)$ and a character $\bar{\theta}$ on $\bar{S}$ such that the character of $\mathcal{X}$ has non-zero inner product with a DL virtual character whose corresponding Deligne--Lusztig parameter is $[s]_x$. Since $[s]$ is strongly regular elliptic, $[s]_x$ is also strongly regular elliptic. Since the finite torus duality is Weyl group equivariant, the strong regularity of $[s]_x$ implies that $\bar{\theta}$ has trivial Weyl stabilizer, i.e. $\bar{\theta}$ is in general position. Therefore, by \cite[Proposition 7.4]{deligne1976representations}, the signed Deligne--Lusztig character is an actual irreducible representation of $\mathbb{G}_x(k)$, which forces it to be the same as $\mathcal{X}$.
    
    By \cite[Proposition 8.2]{deligne1976representations}, if $\bar{S}$ is non-elliptic, then $\bar{S}$ must lie in a proper Levi, and the corresponding signed DL character can be written as the induced representation from a proper Levi of $\mathbb{G}_x(k)$, a contradiction to the cuspidality assumption. Therefore, $\bar{S}$ must be elliptic.

    To prove $\pi$ is extra regular, we choose a lift $S\subset G$ of $\bar{S}$ following \cite[Lemma 3.4.4(2)]{kaletha2019regular}. Then $S$ is a maximal unramified elliptic torus of $G$. The pullback of $\bar{\theta}$ gives a character $\theta_0$ on $S(F)_0$. Let $E/F$ be an unramified field extension such that $S$ splits over $E$ and denote the residue field of $E$ by $k_E$. We need to prove that, for any $g\in N_G(S)(E)$, such that $\alpha:=\mathrm{Ad}(g):S(E)\rightarrow S(E)$ is defined over $F$ and stabilizes $\theta_0$, we must have $g\in S(E)$.

    Let $\bar{\alpha}:\bar{S}_{k_E}\rightarrow \bar{S}_{k_E}$ be the reduction of $\alpha$, which is defined over $k$. Over $E$, choose a conjugation action $i:T_E\xrightarrow{\sim} S_E$, and the reduction gives $\bar{i}:\bar{T}_{k_E}\xrightarrow{\sim}\bar{S}_{k_E}$.

    Transport the automorphism $\alpha$ from $S_E$ to $T_E$:
    $$w_{\alpha}:=i^{-1}\circ\alpha\circ i\in W.$$
    Reduction to $k_E$ also gives us the same Weyl group element. Define a character $\psi$ on $\bar{T}(k_E)$ by
    $$\psi(t)=\bar{\theta}\circ N_{k_E/k}\circ \bar{i}(t).$$
    Since $\alpha$ stabilizes $\theta_0$ and $\bar{\alpha}$ is defined over $k$, we can see that $w_{\alpha}$ stabilizes $\psi$. 

    By \cite[\S 5.3]{deligne1976representations}, $\bar{\theta}$ and $\bar{\theta}\circ N_{k_E/k}$ correspond to the same point $\hat{s}\in \hat{S}$ by our isomorphism $\iota$. Let $\hat{i}:\hat{S}\rightarrow \hat{T}$ be the map induced by $i$, and assume $\hat{t}\in \hat{T}$ is the point corresponding to $\psi$ by $\iota$. Then $\hat{t}$ belongs to $[s]$. The strong regularity of $[s]$ implies that $w_{\alpha}=1$. Therefore, $\alpha$ is the identity map, and $g\in S(E)$. The proof is complete.
\end{proof}

\subsubsection{Regular supercuspidals representation and $L$-packets}

In this part, we further discuss the representations appearing in $\mathcal{R}_{[s]}$.

Let $S\subset G$ denote a maximal unramified elliptic torus. By \cite[Lemma 3.4.3, 3.4.4]{kaletha2019regular}, $S$ gives us a vertex $x_S\in \mathcal{B}(G,F)$ and $S$ reduces to a maximal elliptic torus $\bar{S}\subset \mathbb{G}_{x_S}(k)$. Moreover, every maximal elliptic torus of $\mathbb{G}_{x}(k)$ arises in this way, where $x\in \mathcal{B}(G,F)$ is a vertex.

For a regular depth-zero character $\theta$ on $S(F)$, the restriction of $\theta$ to $S(F)_0$ induces a character $\bar{\theta}$ on $\bar{S}$. Put
$$\kappa_{S,\theta}^0=\epsilon_{\mathbb{G}_{x_S}(k),\bar{S}} R_{\bar{S}}^{\mathbb{G}_{x_S}(k)}(\bar{\theta}), \qquad J_S=S(F)G_{x,0}.$$
Inflate $\kappa_{S,\theta}^0$ to $G_{x,0}$ and take Kaletha's extension $\tilde{\kappa}_{S,\theta}$ to $J_S$ (see \cite[\S 3.4.4]{kaletha2019regular}). By \cite[Lemma 3.4.20]{kaletha2019regular},
$$\pi(S,\theta):=\mathrm{c}\text{-}\mathrm{Ind}_{J_S}^{G(F)}\tilde{\kappa}_{S,\theta}$$
is irreducible supercuspidal. The functions supported on $J_S$ in the compact induction give a copy of the inducing representation. Restricting to $G_{x,0}$ gives
$$\mathrm{Hom}_{G_{x,0}}(\tilde{\kappa}_{S,\theta}|_{G_{x,0}},\pi(S,\theta))\neq 0.$$
Therefore, the DL parameter corresponding to $\pi(S,\theta)$ is only determined by $(\bar{S},\bar{\theta})$. Furthermore, by \cite[Proposition 3.4.27]{kaletha2019regular}, every regular depth-zero supercuspidal representation is of the form $\pi(S,\theta)$ for some pair $(S,\theta)$, and $\pi(S_1,\theta_1)\simeq \pi(S_2,\theta_2)$ if and only if $(S_1,\theta_1)$ and $(S_2,\theta_2)$ are $G(F)$-conjugate. By \Cref{regular}, any $\pi\in \mathcal{R}_{[s]}$ is extra regular, so we can find a pair $(S,\theta)$ such that $\pi=\pi(S,\theta)$.

Note that in our case, the extension $\tilde{\kappa}_{S,\theta}$ is concrete. By \cite[Lemma 7.1.1]{KalethaEndo}, we have
\begin{equation}\label{equation1}
    S(F)=Z(F)\cdot S(F)_0.
\end{equation}
Therefore, $J_S=Z(F)\cdot G_{x,0}$, and the extension is obvious. 

Two pairs $(S,\theta)$ and $(S',\theta')$ are stably conjugate, written as $(S,\theta)\sim_{\mathrm{st}}(S',\theta')$, if some $g\in G(\bar{F})$ induces an $F$-isomorphism
$$c:=\mathrm{Ad}(g)|_S:S\rightarrow S',\qquad \theta'\circ c=\theta.$$
Rational conjugacy requires $g\in G(F)$. 

\begin{lemma}
    Pick $\pi=\pi(S,\theta)\in \mathcal{R}_{[s]}$. Assume $(S',\theta')$ is stably conjugate to $(S,\theta)$. Then $\bar{\Theta}_0(\pi(S',\theta'))=[s]$.
\end{lemma}
\begin{proof}
    Since $S$ and $S'$ are unramified, the $F$-isomorphism $c$ reduces to a $k$-isomorphism $\bar{c}:\bar{S}\rightarrow\bar{S}'$. Choose a finite unramified extension $E/F$ such that $S$ and $S'$ both split over $E$. Over $E$, choose a conjugation isomorphism 
    $$i:T_E\xrightarrow{\sim} S_E,$$
    and define
    $$i':=c\circ i:T_E\xrightarrow{\sim} S'_E.$$
    Let $k_E$ denote the residue field of $E$. Then reduction gives
    $$\bar{i}:\bar{T}_{k_E}\xrightarrow{\sim}\bar{S}_{k_E},\qquad \bar{i}':\bar{T}_{k_E}\xrightarrow{\sim}\bar{S}_{k_E}'$$
    satisfying $\bar{i}'=\bar{c}\circ \bar{i}$.

    Let $\theta_E:=\theta\circ N_{E/F}$ (resp. $\theta_E':=\theta'\circ N_{E/F}$) be the induced character on $S(E)$ (resp. $S'(E)$) and $\bar{\theta}_{k_E}$ (resp. $\bar{\theta}_{k_E}'$) be the corresponding character on $\bar{S}(k_E)$ (resp. $\bar{S}(k_E)'$). Note that $\bar{\theta}_{k_E}=\bar{\theta}\circ N_{k_E/k}$ and $\bar{\theta}_{k_E}'=\bar{\theta}'\circ N_{k_E/k}$. By \cite[\S 5.3]{deligne1976representations}, $\bar{\theta}$ and $\bar{\theta}_{k_E}$ (resp. $\bar{\theta}'$ and $\bar{\theta}_{k_E}'$) have the same corresponding point $\hat{s}$ (resp. $\hat{s}'$) in $\hat{S}$ (resp. $\hat{S}')$. Let $\hat{i}:\hat{S}\rightarrow \hat{T}$ and $\hat{i}':\hat{S}'\rightarrow \hat{T}$ be the maps induced by $i$ and $i'$, respectively. Write $\hat{t}=\hat{i}(\hat{s})\in \hat{T}$ and $\hat{t}'=\hat{i}'(\hat{s}')\in \hat{T}$.

    Since $\bar{c}$ is defined over $k$, it commutes with the Frobenius map. Let $\psi=\bar{\theta}_{k_E}\circ \bar{i}$ and $\psi'=\bar{\theta}_{k_E}'\circ \bar{i}'$ be two characters on $\bar{T}(k_E)$. Then for any $a\in \bar{T}({k_E})$, we have 
    \begin{align*}
        \psi'(a) & = \bar{\theta}_{k_E}'\circ \bar{i}'(a) \\ 
        & = \bar{\theta}'\circ N_{k_E/k}\circ \bar{c} \circ\bar{i}(a) \\
        & = \bar{\theta}'\circ c\circ N_{k_E/k} \circ \bar{i}(a) \\
        & = \bar{\theta}\circ N_{k_E/k}\circ \bar{i}(a) \\
        & = \psi(a).
    \end{align*}
    Therefore, $\psi$ and $\psi'$ denote the same character on $\bar{T}(k_E)$, which implies that $\hat{t}=\hat{t}'$. Therefore, $[s]=\bar{\Theta}_0(\pi(S,\theta))=[\hat{t}]=[\hat{t}']=\bar{\Theta}_0(\pi(S',\theta'))$. The proof is complete.
\end{proof}

By \cite[Lemma 9.6.1]{debacker2009depth}, a Debacker--Reeder $L$-packet can be written as
$$\Pi_{(S,\theta)}^{\mathrm{DR}}=\{\pi(S',\theta')\mid (S',\theta')\sim_{\mathrm{st}}(S,\theta)\}.$$
Therefore, by the above lemma if $\pi=\pi(S,\theta)\in \mathcal{R}_{[s]}$, then any $\pi'\in \Pi_{(S,\theta)}^{\mathrm{DR}}$ also belongs to $\mathcal{R}_{[s]}$. By \cite[Remark 4.9.3]{kaletha2019regular}, every representation in $\mathcal{R}_{[s]}$ arises from the construction of \cite[\S 4.4]{debacker2009depth} and hence belongs to a Debacker--Reeder $L$-packet. Combining all the previous discussion, we can state the main theorem of this section:

\begin{thm}
    The set $\mathcal{R}_{[s]}$ can be written as a disjoint union of Debacker--Reeder $L$-packets.
\end{thm}

\subsubsection{Proof of \Cref{3.1}}\label{section5}

Fix Haar measures $\mathrm{d}g$ on $G(F)$ and $\mathrm{d}z$ on $Z(F)$. Note that for any $\pi(S',\theta')\in \Pi_{(S,\theta)}^{\mathrm{DR}}$, the central character $\omega_{\pi(S',\theta')}=\theta|_{Z}$. Stable conjugacy fixes central elements, so all representations in one Debacker--Reeder packet have the same central character. If $\theta|_{Z}$ is unitary, then by \Cref{equation1}, $\theta$ is also unitary.

Let $\hat{Z}_u$ denote the unitary dual of $Z$. For $\omega\in \hat{Z}_u$, let $\mathcal{H}(G,\omega)$ denote the space of locally constant and compactly supported modulo $Z$ functions, satisfying $h(zg)=\omega(z)^{-1}h(g)$. For $\omega_{\pi}=\omega$, put
$$\pi_{\omega}(h)=\int_{Z\backslash G(F)}h(g)\pi(g)\mathrm{d}g.$$
By \cite[Proposition 3.1.4]{binder2019fields}, we have
\begin{equation}\label{equation2}
h(1)=\int_{\mathrm{Temp}(G,\omega)}\operatorname{tr}\pi_{\omega}(h)\mathrm{d}\mu_{\omega}^{\mathrm{PL}}
\end{equation}
where $\mathrm{Temp}(G,\omega)$ denotes the set of tempered irreducible representations of $G$ with central character $\omega$ and $\mathrm{d}\mu_{\omega}^{\mathrm{PL}}$ denotes the Plancherel measure. For discrete representations $\pi\in \mathrm{Temp}(G,\omega)$, $\mu_{\omega}^{\mathrm{PL}}(\pi)=d(\pi)$, the formal degree of $\pi$.

For any $f\in \mathcal{H}(G)$, define
$$f_{\omega}(g)=\int_Zf(zg)\omega(z)\mathrm{d}z.$$
We can see that $f_{\omega}\in \mathcal{H}(G,\omega)$. For any $\pi$ with central character $\omega$, by \cite[Lemma 8.5.1]{li2026base}, we have $\pi(f)=\pi(f_{\omega})$, and therefore $\operatorname{tr}\pi(f)=\operatorname{tr}\pi(f_{\omega})$. Fourier inversion on $Z(F)$ and \Cref{equation2} gives
\begin{equation}\label{3.1.3.2}
    f(1)=\int_{\hat{Z}_u}\int_{\mathrm{Temp}(G,\omega)}\Theta_{\pi}(f)\mathrm{d}\mu_{\omega}^{\mathrm{PL}}\mathrm{d}\omega.
\end{equation}

Let $E_{[s]}$ be the essentially compact invariant distribution corresponding to $\xi^0(1_{[s]})$. Set
$$D_{[s]}(f)=(\xi^0(1_{[s]})*f)(1).$$
With the usual convention in \cite[\S 3]{haines2014stable}, $D_{[s]}(f)=E_{[s]}(f^{\vee})$, where $f^{\vee}(g):=f(g^{-1})$. The map $f\mapsto f^{\vee}$ preserves the kernel of stable orbital integrals. In other words, it sends unstable functions to unstable functions. Therefore, the stability of $D_{[s]}(f)$ is the same as the stability of $E_{[s]}(f)$.

\Cref{regular} makes the selected spectrum supercuspidal. Therefore, by \Cref{3.1.3.2}, we have
\begin{align}\label{equation3}
    D_{[s]}(f) & =(\xi^0(1_{[s]})*f)(1) \notag \\
    & = \int_{\hat{Z}_u}\sum_{\pi\in \mathcal{R}_{[s],\omega}}d(\pi)\Theta_{\pi}(f)\mathrm{d}\omega,
\end{align}
where $\mathcal{R}_{[s],\omega}$ denotes the set of representations in $\mathcal{R}_{[s]}$ with central character $\omega$. Let $[S,\theta]$ denote the stable conjugacy class of $(S,\theta)$. Previous discussions show that the central character of representations in a Debacker--Reeder packet $\Pi_{(S,\theta)}^{\mathrm{DR}}$ is the same. Denote the central character by $\omega_{S,\theta}$. By \cite[\S 5.3]{debacker2009depth}, all representations in one DeBacker-Reeder $L$-packet have the same formal degree. Denote this common formal degress by $d(s,\theta)$. Therefore, $\mathcal{R}_{[s],\omega}$ is also a disjoint union of Debacker-Reeder $L$-packets. Write $[S,\theta]\in \mathcal{R}_{[s],\omega}$ if $\Pi_{(S,\theta)}^{\mathrm{DR}}\subset \mathcal{R}_{[s],\omega}$. We can rewrite \Cref{equation3} as
$$ D_{[s]}(f)=\int_{\hat{Z}_u}\sum_{[S,\theta]\in \mathcal{R}_{[s],\omega}}d(S,\theta)\mathrm{S}\Theta_{(S,\theta)}(f)\mathrm{d}\omega$$
where
$$\mathrm{S}\Theta_{(S,\theta)}:=\sum_{\pi\in \Pi_{(S,\theta)}^{\mathrm{DR}}}\Theta_{\pi}$$
denotes the stable character of the Debacker--Reeder $L$-packet $\Pi_{(S,\theta)}^{\mathrm{DR}}$. By the stability theorem developed in \cite[\S9-\S12]{debacker2009depth}, the stable character $\mathrm{S}\Theta_{(S,\theta)}$ is a stable distribution. Therefore, when the test function $f$ is unstable, we can see that
$$D_{[s]}(f)=0.$$
The proof is complete.

\subsection{Strongly regular non-elliptic case}

In this part, we study the case where $[s]\in \mathrm{DL}_0$ is only strongly regular but not elliptic. We first state the main theorem of this subsection:

\begin{thm}\label{main2}
    Assume $[s]\in \mathrm{DL}_0$ is strongly regular and non-elliptic. Then $\xi^0(1_{[s]})\in \mathfrak{Z}^0(G)$ is stable.
\end{thm}

\subsubsection{Construction of $\mathcal{R}_{[s]}$}

As above, let $\mathcal{R}_{[s]}$ denote the set of irreducible representations of $G$ such that $\xi^0(1[s])$ acts on them by $1$. Since $[s]$ is strongly regular, there exists a unique element $w_{[s]}\in W$ such that $s^q=w_{[s]}(s)$. 

For a standard Levi subgroup $M\subset G$, define $\mathrm{DL}_0^M=(\hat{T}//W_M)^{[q]}$, and the natural map $j_M:\mathrm{DL}_0^M\rightarrow \mathrm{DL}_0$. Let $\bar{\Theta}_0^M$ denote the map from $\mathrm{Irr}_0(M)$ to $\mathrm{DL}_0^M$.

Choose an irreducible representation $\pi\in\mathcal{R}_{[s]}$. Note that the depth of $\pi$ must be zero. Suppose the supercuspidal support of $\pi$ is $(L,\sigma)$, where $L$ is a standard Levi subgroup and $\sigma$ is an irreducible supercuspidal representation of $L$. In other words, $\pi$ is an irreducible subquotient of $i_P^G( \sigma)$, where $P=LU$ is a standard parabolic subgroup of $G$ containing $L$. By \cite[Theorem 5.2(1)]{moy1996jacquet}, the depth of $\sigma$ is also zero. 

\begin{lemma}\label{lemma1}
    We have the following identity:
    \begin{equation}
        \bar{\Theta}_0(\pi)=j_L(\bar{\Theta}_0^L(\sigma)).
    \end{equation}
\end{lemma}
\begin{proof}
    Denote the corresponding vector space of $\sigma$ by $V$, and use the same symbol $\sigma$ to denote the inflation of $\sigma$ along $P\rightarrow L$. By \cite[Proposition 6.8]{moy1996jacquet}, we can choose a Moy--Prasad type of $(\sigma,V)$ of the form $(x,\mathcal{X})$, where $x\in\mathcal{A}(L,T,F)$ is a vertex and $(\mathcal{X},V_{\mathcal{X}})$ is an irreducible cuspidal representation of $\mathbb{L}_x(k)$. Choose an $L_{x,0}$-equivariant embedding $V_{\mathcal{X}}\hookrightarrow V$, and identify $V_{\mathcal{X}}$ with its image. For any $v\in V_{\mathcal{X}}$ and $l\in L_{x,0}$, we have $\sigma(l)v=\mathcal{X}(l)v$. 

    Note that $\mathcal{A}(L,T,F)=(X_*(T)/X_*(Z(L)^{\circ}))\otimes \mathbb{R}$ and $\mathcal{A}(G,T,F)=(X_*(T)/X_*(Z(G)^{\circ}))\otimes \mathbb{R}$, and there is a natural projection $p_L:\mathcal{A}(G,T,F)\rightarrow \mathcal{A}(L,T,F)$. Choose a lift $\tilde{x}\in \mathcal{A}(G,T,F)$ of $x$ such that $\alpha(\tilde{x})\notin \mathbb{Z}$ for any root $\alpha\in \Phi(G,T)\backslash \Phi(L,T)$. Then the inclusion $L_{x,0}\hookrightarrow G_{\tilde{x},0}$ induces an isomorphism $\mathbb{L}_x(k)=L_{x,0}/L_{x,0+}\xrightarrow{\sim}G_{\tilde{x},0}/G_{\tilde{x},0+}=\mathbb{G}_{\tilde{x}}(k)$. Therefore, $(\mathcal{X},V_{\mathcal{X}})$ can be regarded as an irreducible cuspidal representation of $\mathbb{G}_{\tilde{x}}(k)$. We use the same symbol $(\mathcal{X},V_{\mathcal{X}})$ to denote the inflation to $L_{x,0}$ and $G_{\tilde{x},0}$. 

    For any $v\in V_{\mathcal{X}}$, we define a function $1_v\in i_P^G(\sigma)$ as follows:
    $$1_v(g)=
    \begin{cases}
        \delta_P^{\frac{1}{2}}(p)\sigma(p)\mathcal{X}(h)v, & g=ph\in P(F)G_{\tilde{x},0} \\
        0, & g\notin P(F)G_{\tilde{x},0}.
    \end{cases}
    $$
    We first need to check that $1_v$ is well-defined. Note that $U(F)\cap G_{\tilde{x},0}\subset G_{\tilde{x},0+}$. Moreover, by the Iwahori decomposition, we have $P(F)\cap G_{\tilde{x},0}=L_{x,0}\cdot (U(F)\cap G_{\tilde{x},0})$. For any $g\in P(F)\cap G_{\tilde{x},0}$, write $g=lu$, where $l\in L_{x,0}$ and $u\in U(F)\cap G_{\tilde{x},0}\subset G_{\tilde{x},0+}$. Then we have
    $$\delta_P^{\frac{1}{2}}(g)\sigma(g)v=\sigma(lu)v=\sigma(l)v=\mathcal{X}(l)v=\mathcal{X}(lu)v.$$
    The function is well-defined.

    Now for any $g_0\in G_{\tilde{x},0}$, 
    $$(g_0\cdot 1_v)(g)=1_v(gg_0)=\delta_P^{\frac{1}{2}}(p)\sigma(p)\mathcal{X}(hg_0)v=1_{\mathcal{X}(g_0)v}(g).$$
    The preceding identity shows that $1_v$ is right $G_{\tilde{x},0+}$-invariant, hence smooth. Together with its defining $P(F)$-equivariance, this shows that $1_v\in i_P^G(\sigma)$. Moreover, the map $v\mapsto 1_v$ induces a $G_{\tilde{x},0}$-equivariant embedding $V_{\mathcal{X}}\hookrightarrow i_P^G(\sigma)$. Since $i_P^G(\sigma)$ has finite length and contains the irreducible $G_{\tilde{x},0}$-representation $(\mathcal{X},V_{\mathcal{X}})$, some irreducible subquotient $\pi'$ of $i_P^G(\sigma)$ contains $(\tilde{x},\mathcal{X})$ as a Moy--Prasad type. By \cite[Proposition 7.11]{bhattacharya2025description}, and the identification of reductive quotients above,
    $$\bar{\Theta}_0(\pi')=j_L(\bar{\Theta}_0^L(\sigma))=:[s'].$$
    Thus $\xi^0(1_{[s']})$ acts on $\pi'$ by $1$. By \cite[\S 3.3.2]{haines2014stable}, it acts by a scalar on $i_P^G(\sigma)$; this scalar must therefore be $1$. In particular, it acts on $\pi$ by $1$, so
    $$\bar{\Theta}_0(\pi)=\bar{\Theta}_0(\pi')=j_L(\bar{\Theta}_0^L(\sigma)).$$
    The proof is complete.
\end{proof}

\begin{lemma}\label{lemma2}
    If $(\sigma,V)$ is depth-zero supercuspidal on $L(F)$ and $j_L(\bar{\Theta}_0^L(\sigma))=[s]$, then $\bar{\Theta}_0^L(\sigma)\in \mathrm{DL}_0^L$ is strongly regular elliptic.
\end{lemma}
\begin{proof}
    As above, choose a Moy--Prasad type $(x,\mathcal{X})$ of $\sigma$, with $x$ a vertex of $\mathcal{B}(L,F)$. Applying the finite Deligne--Lusztig argument in the proof of \Cref{regular}, we obtain a maximal elliptic torus $\bar{S}\subset \mathbb{L}_x(k)$ and a character $\bar{\theta}$ on $\bar{S}$ such that 
    $$\mathcal{X}=\epsilon_{\mathbb{L}_x(k),\bar{S}} R_{\bar{S}}^{\mathbb{L}_x(k)}(\bar{\theta}).$$
    Let $s_L\in \hat{T}$ denote the element corresponding to $(\bar{S},\bar{\theta})$. The strong regularity of $[s]$ implies that $\bar{\Theta}_0^L(\sigma)$ is strongly regular. Then there exists a unique element $w_s\in W_L$ such that $s_L^q=w_s(s_L)$. Write $\bar{\Theta}_0^L(\sigma)=[s_L]_L$.

    Lift $\bar{S}$ to a maximal unramified elliptic torus $S\subset L$ by \cite[Lemma 3.4.4(2)]{kaletha2019regular}. Let $E/F$ be a finite unramified field extension such that $S$ splits over $E$, and denote its residue field by $k_E$. Let $\gamma\in \mathrm{Gal}(E/F)$ be the generator that induces the Frobenius map on $k_E/k$. Choose $h\in L(E)$ such that there is an isomorphism $i:=(\operatorname{Ad}(h):T(E)\rightarrow S(E))$. Since both $S$ and $T$ are defined over $F$, the element $h^{-1}\gamma(h)\in N_L(T)(E)$. Denote its image in the Weyl group by $w$. The isomorphism $i$ induces an isomorphism $\bar{i}:\bar{T}(k_E)\rightarrow \bar{S}(k_E)$. The character $\bar{\theta}$ on $\bar{S}$ can therefore be transferred to a character $\bar{\theta}_T$ on $\{t\in \bar{T}(k_E)\mid t^q=w^{-1}(t)\}$. Define $\bar{\theta}_E=\bar{\theta}\circ N_{k_E/k}$. By \cite[\S 5.3]{deligne1976representations}, $\bar{\theta}$ and $\bar{\theta}_E$ map to the same element $s_L\in \hat{T}$. Let $\bar{\theta}_{T,E}:=\bar{\theta}_E\circ\bar{i}$. Then we have $\bar{\theta}_{T,E}(t)=\bar{\theta}_{T,E}(w(t^q))$ for any $t\in \bar{T}(k_E)$. In other words, $\bar{\theta}_{T,E}(w^{-1}(t))=\bar{\theta}_{T,E}(t^q)$. Therefore, $s_L^q=w(s_L)$. The strong regularity of $s_L$ implies that $w=w_s$.

    Let $Y=X_*(T)\otimes \mathbb{R}=X_*(T_E)\otimes \mathbb{R}$ and $Y_S=X_*(S_E)\otimes \mathbb{R}$. The isomorphism $i$ induces an isomorphism $i_*:Y\rightarrow Y_S$. For any $\lambda\in Y$, we have
    $$\gamma\circ i_*(\lambda)=i_*(w\lambda).$$
    If $\lambda\in Y^w$, then since $S$ is elliptic,
    $$i_*(\lambda)\in Y_S^{\gamma}=X_*(Z(L)^{\circ})\otimes \mathbb{R}.$$
    Therefore,
    $$Y^w=X_*(Z(L)^{\circ})\otimes \mathbb{R}$$
    and this implies the ellipticity of $[s_L]_L$. The proof is complete.
\end{proof}

Define
$$W_{\sigma}:=\{nL(F)\in N_G(L)(F)/L(F):{^n\sigma}\simeq \sigma\otimes \chi \textnormal{ for some }\chi\in X_{\mathrm{nr}}(L)\}.$$
We have the following lemma:
\begin{lemma}\label{lemma3}
    Suppose $j_L(\bar{\Theta}_0^L(\sigma))=[s]$. Then $W_{\sigma}=1$. In particular, $\pi=i_P^G(\sigma)$.
\end{lemma}
\begin{proof}
    Fix an element $nL(F)\in W_{\sigma}$, choose $n$ such that $n\in N_G(T)(F)$, and denote its image in $W$ by $w_n$. Choose a Moy--Prasad type $(x,\mathcal{X})$ and an actual point $s_L\in\hat{T}$ corresponding to $(x,\mathcal{X})$. Since $n$ normalizes $L(F)$, it also normalizes $Z(L)$. Choose a lift $\tilde{x}\in \mathcal{A}(G,T,F)$ of $x$. Since $n\in N_G(T)(F)$, the point $n\tilde{x}$ belongs to $\mathcal{A}(G,T,F)$. Denote its projection back to $\mathcal{A}(L,T,F)$ by $nx$. Note that the projection does not depend on the choice of $\tilde{x}$. We have
    $$L_{nx,0}=G_{n\tilde{x},0}\cap L(F)=nL_{x,0}n^{-1}, \qquad L_{nx,0+}=G_{n\tilde{x},0+}\cap L(F)=nL_{x,0+}n^{-1}.$$
    Therefore, $n$ induces a $k$-isomorphism $\bar{n}:\mathbb{L}_x(k)\rightarrow \mathbb{L}_{nx}(k)$. If we write $\mathcal{X}=\epsilon_{\mathbb{L}_x(k),\bar{S}} R_{\bar{S}}^{\mathbb{L}_x(k)}(\bar{\theta})$, then the pair $(nx,\bar{n}\mathcal{X})$, where 
    $$\bar{n}\mathcal{X}:=\epsilon_{\mathbb{L}_{nx}(k),\bar{n}\bar{S}} R_{\bar{n}\bar{S}}^{\mathbb{L}_{nx}(k)}(\bar{n}\bar{\theta}),$$
    will be a Moy--Prasad type for $^n\sigma$. Since $n\in N_G(T)(F)$, 
    the element $w_n(s_L)$ represents the pair $(\bar{n}\bar{S},\bar{n}\bar{\theta})$ and thus $\bar{\Theta}_0^L({^n\sigma})=[w_n(s_L)]_L$.

    By definition of $W_{\sigma}$, we have $^n\sigma\simeq\sigma\otimes \chi$ for some $\chi\in X_{\mathrm{nr}}(L)$. Since $\chi$ is trivial on all parahoric subgroups, we have
    $$\bar{\Theta}_0^L(\sigma)=\bar{\Theta}_0^L(\sigma\otimes \chi)=[s_L]_L.$$
    Therefore, $w_n(s_L)\in [s_L]_L$, which implies that $w_n\in W_L$. This shows that $n\in N_L(T)(F)\subset L(F)$, so $W_{\sigma}=1$. By \cite[Corollary 2]{roche2002parabolic}, $i_P^G(\sigma)$ is irreducible, and therefore $\pi=i_P^G(\sigma)$. The proof is complete.
\end{proof}

We now obtain a description of the set $\mathcal{R}_{[s]}$:
\begin{thm}\label{R_s non elliptic}
    Any representation $\pi\in \mathcal{R}_{[s]}$ can be written as $i_P^G(\sigma)$, where $P=LU$ is a standard parabolic subgroup and $\sigma$ is a depth-zero extra regular supercuspidal representation of $L$ such that $j_L(\bar{\Theta}_0^L(\sigma))=[s]$. Moreover, fix a standard Levi subgroup $L$ and let $P=LU$ be the standard parabolic subgroup containing $L$. Moreover, define
    $$\mathcal{R}_{[s],L}:=\{\pi\in \mathcal{R}_{[s]}\mid \pi\simeq i_P^G(\sigma)\textnormal{ for some irreducible supercuspidal }\sigma\textnormal{ of }L(F)\}.$$
    Then we have
    \begin{equation}
        \mathcal{R}_{[s],L}=\bigcup_{[s_L]_L\in j_L^{-1}([s])_{\mathrm{ell}}}\{i_P^G(\sigma)\mid \sigma\in \mathcal{R}_{[s_L]_L}\},
    \end{equation}
    where $j_L^{-1}([s])_{\mathrm{ell}}$ denotes the elliptic preimage of $[s]$ under the map $j_L$.
\end{thm}
\begin{proof}
    For any $\pi\in \mathcal{R}_{[s]}$, there exists a standard Levi subgroup $L\subset G$ and a depth-zero supercuspidal representation $(\sigma,V)$ of $L(F)$ such that $\pi$ is an irreducible subquotient of $i_P^G(\sigma)$, where $P$ is the standard parabolic subgroup containing $L$. By \Cref{lemma1} and \Cref{lemma3}, we can see that $j_L(\bar{\Theta}_0^L(\sigma))=\bar{\Theta}_0(\pi)=[s]$ and $i_P^G(\sigma)$ is irreducible. Therefore, $\pi=i_P^G(\sigma)$. Moreover, by \Cref{lemma2}, $\bar{\Theta}_0^L(\sigma)$ is strongly regular elliptic, and the fact that $\sigma$ is extra regular supercuspidal follows from \Cref{regular}. This proves the first part.

    For the second part, if $\pi\simeq i_P^G(\sigma)$ for a supercuspidal representation $\sigma$ of $L(F)$, then by \Cref{lemma1} and \Cref{lemma2}, $j_L(\bar{\Theta}_0^L(\sigma))=[s]$ and $\bar{\Theta}_0^L(\sigma)$ is strongly regular elliptic. Hence $\pi\in \bigcup_{[s_L]_L\in j_L^{-1}([s])_{\mathrm{ell}}}\{i_P^G(\sigma)\mid \sigma\in \mathcal{R}_{[s_L]_L}\}$. Conversely, choose $[s_L]_L\in j_L^{-1}([s])_{\mathrm{ell}}$. By \Cref{regular}, any $\sigma\in \mathcal{R}_{[s_L]_L}$ is depth-zero extra regular supercuspidal. \Cref{lemma1} and \Cref{lemma3} imply that $\pi:=i_P^G(\sigma)$ is irreducible and $\bar{\Theta}_0(\pi)=j_L(\bar{\Theta}_0^L(\sigma))=[s]$, and hence $\pi\in \mathcal{R}_{[s],L}$. The proof is complete.
\end{proof}

\subsubsection{Proof of \Cref{main2}}

In this subsection, we finish the proof of \Cref{main2}.

Let $\mathcal{L}$ be a set of standard representatives for the $G(F)$-conjugacy classes of Levi subgroups. For each $L\in \mathcal{L}$, fix the standard parabolic subgroup $P_L=LU_L$, and put
$$A_L=Z(L)^{\circ},\qquad W^G(L)=N_G(L)(F)/L(F).$$
Let $\hat{A}_L^u$ denote the set of unitary characters on $A_L(F)$, and for each $\zeta\in \hat{A}_L^u$, write
$$\mathrm{Irr}_2(L)_{\zeta}:=\{\sigma\in\mathrm{Irr}_2(L)\mid \omega_{\sigma}|_{A_L(F)}=\zeta\},$$
where $\mathrm{Irr}_2(L)$ denotes the set of discrete representations of $L(F)$. For any $f\in \mathcal{H}(G)$, by \cite[Lemma 4.2.6]{varma2024some}, we have
\begin{equation}\label{Plancherel}
    f(1)=\sum_{L\in \mathcal{L}} b_{G,L}\int_{\hat{A}_L^u}\sum_{\sigma\in \mathrm{Irr}_2(L)_{\zeta}}d_L(\sigma)\mu_{G,L}(\sigma)\Theta_{i_{P_L}^G(\sigma)}(f)\mathrm{d}\zeta
\end{equation}
where $b_{G,L}$ is a constant that depends only on $G$ and $L$ and $d_L(\sigma)$ denotes the formal degree of $\sigma$. Define
$$D_{[s]}(f):=(\xi^0(1_{[s]})*f)(1).$$
Following \Cref{section5}, we only need to prove $D_{[s]}(f)=0$ for any unstable function $f\in \mathcal{H}(G)$.

We have the following lemma:
\begin{lemma}\label{no discrete}
    Let $L\subset G$ be a standard Levi subgroup, and $\sigma\in \mathrm{Irr}_2(L)$. Suppose that $i_{P_L}^G(\sigma)$ has an irreducible subquotient $\pi\in \mathcal{R}_{[s]}$, where $[s]$ is strongly regular. Then $\sigma$ is depth-zero supercuspidal. Moreover,
    $$j_L(\bar{\Theta}_0^L(\sigma))=[s]$$
    and $i_{P_L}^G(\sigma)$ is irreducible.
\end{lemma}
\begin{proof}
    Choose a standard Levi subgroup $M\subset L$ and a supercuspidal representation $\rho$ of $M(F)$, such that $\sigma$ is an irreducible subquotient of $i_{P_M^L}^L(\rho)$, where $P_M^L$ denotes the standard parabolic subgroup of $L$ containing $M$. By the transitivity of normalized induction, we can see that $\pi$ is an irreducible subquotient of $i_{P_M}^G(\rho)$. By \Cref{lemma3}, we can see that $\pi=i_{P_M}^G(\rho)$ and $\sigma=i_{P_M^L}^L(\rho)$. Since a full induction from a proper Levi cannot be discrete, we deduce that $M=L$, and $\pi=i_P^G(\sigma)\in \mathcal{R}_{[s],L}$. The proof is complete.
\end{proof}

Choose a hyperspecial maximal compact subgroup $K\subset G(F)$ in good position with the Borel subgroup $B$ that we fix. For any $f\in \mathcal{H}(G)$, we define the constant term $f_{P_L}$ as follows:
$$f_{P_L}(l)=\delta_{P_L}^{\frac{1}{2}}(l)\int_{U_L}\int_Kf(k^{-1}luk)\mathrm{d}k\mathrm{d}u$$
where the Haar measure $\mathrm{d}k$ is chosen such that $\mathrm{vol}_{\mathrm{d}k}(K)=1$ and the Haar measure $\mathrm{d}u$ is chosen such that $\mathrm{vol}_{\mathrm{d}u}(K\cap U_L)=1$. We have the following lemma:
\begin{lemma}\label{lemma4}
    The map $f\mapsto f_{P_L}$ sends unstable functions on $G(F)$ to unstable functions on $L(F)$.
\end{lemma}
\begin{proof}
    Choose an arbitrary unstable function $f\in \mathcal{H}(G)$. Suppose $\gamma\in L(F)$ is strongly regular in $G$. Then we have
    $$T:=G_{\gamma}=L_{\gamma}$$
    is a maximal torus in $L$. Put $d_{G/L}(\gamma)=|\operatorname{det}(1-\operatorname{Ad}\gamma\;;\;\mathfrak{g}/\mathfrak{l})|$. Then by \cite[Proposition 4.4.9]{laumon1996cohomology}, we have
    $$\operatorname{O}_{\gamma}^L(f_{P_L})=d_{G/L}(\gamma)^{1/2}\operatorname{O}_{\gamma}^G(f).$$

    For $H=L, G$, the rational conjugacy classes within the stable conjugacy class of $\gamma$ is parametrized by
    $$D_H(T):=\mathrm{ker}[H^1(F,T)\rightarrow H^1(F,H)].$$
    Note that $D_L(T)$ is naturally embedded in $D_G(T)$. By \cite[Equation 4.1.1]{haines2009base}, we have $\mathrm{ker}[H^1(F,L)\rightarrow H^1(F,G)]=\{1\}$. Therefore, $D_L(T)=D_G(T)$ and 
    $$\operatorname{SO}_{\gamma}^L(f_{P_L})=d_{G/L}(\gamma)^{1/2}\operatorname{SO}_{\gamma}^G(f)=0.$$

    Next we assume $\gamma\in L(F)$ is strongly regular in $L$ but not in $G$. Put $T=L_{\gamma}$, and let $W(G,T)=N_G(T)(\bar{F})/T(\bar{F})$ denotes the absolute Weyl group. Choose a compact open subgroup $J\subset L(F)$ such that $f_{P_L}$ is left $J$-invariant. Let $A_L$ be the maximal $F$-split torus in $Z(L)$. We have 
    $$Z(L)\subset T,\qquad L=Z_G(A_L).$$

    For every $1\neq w\in W(G,T)$, consider
    $$Z_w:=\{z\in A_{L,\bar{F}}:w(z\gamma)=z\gamma\}.$$
    This is a Zariski closed subset of $A_{L,\bar{F}}$.

    We claim that $Z_w$ is proper. Otherwise, setting $z=1$ gives $w(\gamma)=\gamma$. Substituting this back shows that $w$ fixes $A_L$ pointwise. A representative $n\in N_G(T)(\bar{F})$ of $w$ therefore belongs to $Z_G(A_L)(\bar{F})=L(\bar{F})$. This contradicts the strong regularity of $\gamma$ in $L$.

    The set $A_L(F)\cap J$ is Zariski dense in $A_{L,\bar{F}}$. Since $A_{L,\bar{F}}$ is irreducible, the finite union $\bigcup_{1\neq w\in W(G,T)}Z_w$ cannot contain $A_L(F)\cap J$. Therefore, we can always find $z\in A_L(F)\cap J$ such that $z\gamma$ is strongly regular in $G$.

    Because $z$ is central in $L(F)$, multiplication by $z$ gives a bijection between the rational classes in the stable $L$-classes of $\gamma$ and $z\gamma$. Their corresponding centralizers are identical. Furthermore, $f_{P_L}(zl)=f_{P_L}(l)$ for any $l\in L(F)$. Therefore, $$\operatorname{SO}_{\gamma}^L(f_{P_L})=\operatorname{SO}_{z\gamma}^L(f_{P_L})=0$$
    by the first case. In conclusion, for any strongly regular semisimple element $\gamma\in L(F)$, 
    $$\operatorname{SO}_{\gamma}^L(f_{P_L})=0.$$
    This shows that $f_{P_L}$ is unstable. The proof is complete.
\end{proof}

By \cite[Theorem 2]{van1972computation}, $\Theta_{i_{P_L}^G(\sigma)}(f)=\Theta_{\sigma}(f_{P_L})$. Therefore, for the function $\xi^0(1_{[s]})*f$, we can rewrite \Cref{Plancherel} as
\begin{equation}
    (\xi^0(1_{[s]})*f)(1)=\sum_{L\in \mathcal{L}} b_{G,L}\int_{\hat{A}_L^u}\sum_{\sigma\in \mathrm{Irr}_2(L)_{\zeta}}d_L(\sigma)\mu_{G,L}(\sigma)\xi^0(1_{[s]})(i_{P_L}^G(\sigma))\Theta_{\sigma}(f_{P_L})\mathrm{d}\zeta.
\end{equation}
By \Cref{R_s non elliptic} and \Cref{no discrete}, the set of representations $\sigma\in \mathrm{Irr}_2(L)_{\zeta}$ such that $\xi^0(1_{[s]})(i_{P_L}^G(\sigma))\neq 0$ is the same as $\bigsqcup_{[s_L]_L\in j_L^{-1}([s])_{\mathrm{ell}}}\mathcal{R}_{[s_L]_L, \zeta}$. Therefore, if we fix a standard Levi subgroup $L\in \mathcal{L}$, then
\begin{align}\label{1}
    b_{G,L}\int_{\hat{A}_L^u}\sum_{\sigma\in \mathrm{Irr}_2(L)_{\zeta}}&d_L(\sigma)\mu_{G,L}(\sigma)\xi^0(1_{[s]})(i_{P_L}^G(\sigma))  \Theta_{\sigma}(f_{P_L})\mathrm{d}\zeta= \notag \\ &b_{G,L}\sum_{[s_L]_L\in j^{-1}_L([s])_{\mathrm{ell}}}\int_{\hat{A}_L^u}\sum_{\sigma\in \mathcal{R}_{[s_L]_L,\zeta}}d_L(\sigma)\mu_{G,L}(\sigma)\Theta_{\sigma}(f_{P_L})\mathrm{d}\zeta.
\end{align}
By \cite[Corollary 4.2.13(1)]{varma2024some}, $\mu_{G,L}$ is constant on a Debacker--Reeder $L$-packet. The proof in \Cref{section5}, together with \Cref{lemma4}, shows that 
$$\int_{\hat{A}_L^u}\sum_{\sigma\in \mathcal{R}_{[s_L]_L,\zeta}}d_L(\sigma)\mu_{G,L}(\sigma)\Theta_{\sigma}(f_{P_L})\mathrm{d}\zeta=0$$
for any unstable function $f\in \mathcal{H}(G)$, which implies that \Cref{1}$=0$. Let $L$ vary. We get
$$D_{[s]}(f)=0$$
for any unstable function $f$. The proof is complete.

\section{The positive depth case}\label{sec:Positivedepths}

In this section we prove the positive depth cases of \cite[Conjecture 1.1]{bhattacharya2025description}. The proof in the positive depth cases is uniform, and is motivated by the approach for positive depth projectors in \cite{bkv2016depthrprojector}. We use a theorem of Waldspurger \cite{waldspurger1997lemme} to reduce the question of stability of distributions to a question of stability of certain subsets in the dual of the Lie algebra $\fg^*$ in the sense of \cite[8.3(d)]{bkv2016depthrprojector}.  

Fix a uniformizer $\varpi \in \mf O$. Let $\psi : F\rightarrow\C^\times $ be an additive character, trivial on $\varpi\mf O$ but non-trivial on $\mf O$. This fixes a non-trivial character $\tilde\psi:\F_q\rightarrow \C^\times $ and we use it for defining Fourier transforms for Moy-Prasad quotients (\cite[Section 4.2]{bhattacharya2025description}) and constructing stable functions on them.

\subsection{Review of a few constructions and definitions}
We briefly review some constructions and definitions following \cite{bhattacharya2025description}. Let $r \in \frac{1}{m}\Z_{\geq 0}$  with $p \nmid m$. Similar to \cite{bhattacharya2025description}, for all such $m$, we fix $\nu_m \in \mf O_{F^t}$ with $\mathrm{v}(\nu_m)=\frac{1}{m}$. We fix a chamber $\mcC \subset \mathcal{A}(T,F)\subset  \B(G,F)$. Let $\B_m$ ($\mc A_m, \bar \mcC_m$) be the building (resp. apartment, closure of the chamber) with refined simplicial structure as explained in the introduction and \cite[Section 2.4]{bhattacharya2025description}. Let $[\B_m]$ ($[\mc A_m], [\bar{\mcC}_m]$ denote the set of refined facets (or open polysimplices) in $\B_m$ (resp. $\mc A_m, \bar{\mcC}_m$)
 Let $E$ be a totally ramified extension of $F^{nr}$ of ramification degree $m$. For $x \in \B(G,F)$, let $\mbL_x=(\G_E)_x$ denote the $\bar \F_q$ group given by $\mbL_x(\bar \F_q)= G(E)_{x,0}/G(E)_{x,0+}$. Similarly let $\bl_{x,r}=(\bg_E)_{x,r}:=\fg(E)_{x,r}/\fg(E)_{x,r+}$, which is a $\bar\F_q$-vector space. Note that there exists an isomorphism $\bl_x=\mathrm{Lie}(\mbL_x)\cong \bl_{x,r}$ which depends on choice of $\nu_m$. Similarly we can define $\bl^*_{x,r}=\frac{\fg(E)^*_{x,-r}}{\fg(E)^*_{x,-r+}}$. These $\bar\F_q$-spaces have an $\F_q$-structure depending on choice of $\nu_m$ (for example check \cite[Section 4.3]{bhattacharya2025description}), and we make analogous definitions for $\mbL_\sigma, \bl_\sigma$ etc. for $\sigma \in [\B_m]$.

 \paragraph{} For $\sigma \in [\mc A_m]$, we have maps of $\bar\F_q$-varieties which descend to maps over $\F_q$ (\cite[Lemma 4.8]{bhattacharya2025description}), and hence we have a map of $\F_q$-points
 \begin{equation}\label{stableGIT}
    \begin{tikzpicture}[baseline=(current  bounding  box.center)]
        \node (A) at (0,0) {$ \bg_{\sigma,r}^*$};
        \node (B) at (2,0) {$(\bg_E)_{\sigma,r}^*(\F_q) $};
        \node (D) at (4,0) {$\bl_\sigma^*(\F_q)$};
        \node (C) at (6,0) {$\bl_\sigma^*//\mbL_\sigma(\F_q) $};
        \node (E) at (3,-2) {$\bg_{\sigma,r}^*//\G_\sigma(\F_q)$};
        \node (F) at (8.8,0) {$\bt^*//W_\sigma(\F_q)$};
        \draw[->] (A)--(B);
         \draw[->] (B)edge node[above] {$\simeq$}(D);
        \draw[->] (D)--(C);
         \draw[->] (A)--(E);
         \draw[->] (E)--(C);
          \draw[->] (C)edge node[above] {$\simeq$}(F);
    \end{tikzpicture}
\end{equation}
where the last horizontal isomorphism comes from Chevalley isomorphism and $W_\sigma$ is the absolute Weyl group of $\mbL_\sigma$, which is the finite Weyl group corresponding to the refined facet $\sigma$. For further details about the maps, check \cite[Section 4.3]{bhattacharya2025description}. The composition of the horizontal maps $\bg_{\sigma,r}^*\rightarrow \bt^*//W_\sigma(\F_q)$ is defined to be $\chi_{\sigma,r}$.

\paragraph{} The pairing $\langle\cdot,\cdot\rangle:\fg^*\times \fg\rightarrow F$ induces a $\G_\sigma(\F_q)$-invariant pairing 
\begin{align*}
  \frac{\fg^*_{\sigma,-r}}{\fg^*_{\sigma, -r+}} \times   \frac{\fg_{\sigma,r}}{\fg_{\sigma,r+}} &\rightarrow \F_q\\
  (X,Y) & \mapsto \langle X,Y\rangle \:\mathrm{mod}\:\varpi \mf O
\end{align*}
of $\F_q$-vector spaces. Let $C(\bg_{\sigma,r})$ denote the space of $\G_\sigma(\F_q)$-invariant functions on $\bg_{\sigma,r}$ equipped with the convolution product
\begin{equation}
    f*f'(X)=\left|\bg_{\sigma,r}\right|^{-1/2}\sum_{Y\in \bg_{\sigma,r}}f(X-Y)f'(Y).
\end{equation}
and $C\left(\bg_{\sigma,r}^*\right)$ denote the space of  $\G_\sigma(\F_q)$-invariant functions on $\bg_{\sigma,r}^*$ equipped with point-wise multiplication. We have the Fourier transform $\mcF_{\bg_{\sigma,r}}:C(\bg_{\sigma,r}) \rightarrow C\left(\bg_{\sigma,r}^*\right)$ given by
\begin{equation*}\label{fourierdefnlie}
    \mcF_{\bg_{\sigma,r}}(f)(X^*)=\left|\bg_{\sigma,r}\right|^{-1/2}\sum_{Y\in \bg_{\sigma,r}^F}\Tilde\psi( (X^*,Y))f(Y).
\end{equation*}
and similarly $\mcF_{\bg_{\sigma,r}^*}:C\left(\bg_{\sigma,r}^*\right) \longrightarrow C(\bg_{\sigma,r}) $ following \cite[Section 4.2]{bhattacharya2025description}. We have the following definition
\begin{defn}[Definition 4.11, \cite{bhattacharya2025description}]
     We define $f \in C(\bg_{\sigma,r})$ to be stable if $\mcF_{\bg_{\sigma,r}}(f):\bg_{\sigma,r}^* \rightarrow \C$ factors through 
    \[
    \bg_{\sigma,r}^* \xlongrightarrow{\chi_{\sigma,r}} \bt^*//W_\sigma(\F_q) \longrightarrow \C.
    \]
   We use $C^{\mathrm{st}}(\bg_{\sigma,r})$ to denote stable functions on $\bg_{\sigma,r}^F$. Using properties of Fourier transforms, we immediately see that $C^{\mathrm{st}}(\bg_{\sigma,r}) \subset C(\bg_{\sigma,r})$  forms a subalgebra with unit $|\bg_{\sigma,r}|^{1/2}\mathbbm{1}_0$. 
\end{defn}
Given $\varepsilon \in \C[\bt^*//W_\sigma(\F_q)]$, let $f_\varepsilon:= \mcF_{\bg_{\sigma,r}^*}(\chi_{\sigma,r}^*(\varepsilon)^-)= \mcF_{\bg_{\sigma,r}^*}^{-1}(\chi_{\sigma,r}^*(\varepsilon))$. From \cite[Prop. 4.13]{bhattacharya2025description}, we know that $f_\varepsilon \in C^{\mathrm{st}}(\bg_{\sigma,r})$ and the map 
\begin{equation*}
    \rho_{\sigma,r}: \C[\bt^*//W_\sigma(\F_q)] \longrightarrow C^{\mathrm{st}}(\bg_{\sigma,r})
\end{equation*}
is an algebra homomorphism. 

\paragraph{} Let $\sigma \in [\bar \mcC_m]$. Note that $W_\sigma $ agrees with the image of $N_{G(E)_{\sigma,0}}(T)$ in $W$ (see \cite{debtotallyram}, Lemma 7.2.1), and hence there is an embedding $W_\sigma \hookrightarrow W$. Thus, we have a map 
$t_{\sigma} : \bt^*//W_\sigma(\F_q) \rightarrow \bt^*//W(\F_q)$
and hence a map 
\begin{equation*}
    t_{\sigma}^*:\C[\bt^*//W(\F_q)]\longrightarrow\C[\bt^*//W_\sigma(\F_q)]
\end{equation*}

For $\overline{(\nu, \vartheta)} \in \DL_r$ (as defined \cite[Section 6.1]{bhattacharya2025description}, we have $\vartheta \in \bt^*//W(\F_q)$ and since we have fixed $\nu_m$, we can work with $\vartheta$ and assume $\vartheta \in \DL_r$ without loss of generality. Further, we have algebra maps (depending on choice of $\nu_m$) $j_{\sigma,r}:\C[\bt^*//W(\F_q)] \rightarrow \mc M_\sigma^r$ in \cite[Section 5]{bhattacharya2025description} which are compatible with the inverse system maps between $\mc M_\sigma^r$, and we give more details about these maps in the next section. The maps $j_{\sigma,r}$ induce maps 
\begin{equation*}
    \xi^r:\C[\bt^*//W(\F_q)] \longrightarrow\mf Z^r(G)
\end{equation*}
via the inverse limit description of $\mf Z^r(G)$, and $E_{r,\vartheta} =\xi^r(1_\vartheta)$ are projector elements in the center such that 
  \begin{equation*}
        E_{r,\vartheta}|_\pi=\begin{cases}
        \mathrm{Id}, & \text{ if  $\pi\in \Irr(G)_r$ and }   \bar\Theta_r(\pi)=\vartheta\\
        0, & \text{ otherwise}
    \end{cases}
    \end{equation*}

In this section, we prove the following theorem 
\begin{thm}
    Let $r \in \Z_{(p)}\cap\Q_{>0}$ and $\xi^r$ be as above. We have $\mathrm{Im}(\xi^r)\subset \mf Z^{\mathrm{st}}(G)\cap \mf Z^r(G)$.
\end{thm}
As explained in \cite[Section 9]{bhattacharya2025description}, it is enough to show that the projectors $E_{r,\vartheta}$ for $\vartheta \in \DL_r$ are stable.

\subsection{Fourier Transforms, Quasi-logarithms and stability}

For any $\vartheta\in \DL_r$, we will construct an open, closed stable $\Ad^*(G(F))$-invariant subset $\fg^*_{-r,\vartheta} \subset \fg^*_{-r}$, where $\Ad^*$ denotes the co-adjoint action. For $ \sigma \in [\bar{\mcC}_m]$, we have  maps 
\begin{equation*}
   J_{\sigma,r}:  \fg^*_{\sigma,-r } \rightarrow \frac{ \fg^*_{\sigma,-r }}{ \fg^*_{\sigma,-r+}}\xrightarrow[]{\chi_{\sigma,r}}\bt^*//W_\sigma(\F_q)\xrightarrow{t_\sigma}\bt^*//W(\F_q)
\end{equation*}
with $\chi_{\sigma,r}$ and $t_\sigma$ as defined in previous section (see \cite[Section 5]{bhattacharya2025description} for more details).  Define $\fg^*_{\sigma ,-r,\vartheta}=J_{\sigma,r}^{-1}(\vartheta) $. For $\sigma' \in [\B_m]$ such that $\sigma'=g\cdot \sigma$ for $g \in G(F)$ and $\sigma \in [\bar{\mcC}_m]$, we can define $\fg^*_{\sigma' ,-r,\vartheta}= \Ad^*(g)\fg^*_{\sigma ,-r,\vartheta}$ and this is well defined. 
Let $\fg^*_{-r,\vartheta} = \cup_{\sigma \in [\B_m]}\fg^*_{\sigma ,-r,\vartheta}$. Note that $\fg^*_{-r,\vartheta} = \cup_{\sigma \in [\bar{\mcC}_m]} \Ad^*(G(F))\fg^*_{\sigma ,-r,\vartheta}$ and hence it is $G(F)$-invariant by definition, and open. We prove that it is closed and stable. Note that for different $\vartheta,\vartheta'\in \mathrm{DL}_r$, $\mathfrak{g}^*_{-r,\vartheta}\cap  \mathfrak{g}^*_{-r,\vartheta'}=\varnothing$ by \cite[Lemma 36]{chen2025langlands}. Therefore, we have
$$\fg^*\setminus \fg^*_{-r,\vartheta}= (\fg^*\setminus \fg^*_{-r}) \bigcup_{\vartheta \neq \vartheta'} (\cup \fg^*_{-r,\vartheta'}).$$ 
It is shown in \cite[Lemma B.3]{bkv2016depthrprojector} that $\fg^*_{-r}$ is closed, and hence $\fg^*\setminus \fg^*_{-r,\vartheta}$ is open and $\fg^*_{-r,\vartheta}$ is closed.

\paragraph{} Let $\fg_{sr}^*\subset \fg^*$ be the set of strongly regular semisimple elements of $\fg^*$ (see \cite[Sec 8.1(b)]{bkv2016depthrprojector}). To show that $\fg^*_{-r,\vartheta}$ is stable, we need to show that for every $G(\bar F)$-conjugate $b,b'\in \fg^*_{sr}$ such that $b \in \fg^*_{-r,\vartheta}$, we have $b' \in \fg^*_{-r,\vartheta}$.
Since $b, b' $ are $G(\bar F)$-conjugate, and $F^{nr}$ is of cohomological dimension one, we can conclude that $b,b'$ are $G(F^{nr})$-conjugate, and hence $G(F^\flat)$-conjugate for some finite unramified extension $F^\flat /F$. \cite[Lemma 24 and Remark 26]{chen2025langlands} shows that the Deligne-Lusztig parameter attached to a Moy-Prasad type does not change on moving to a finite unramified extension, and hence $b' \in \fg^*_{-r,\vartheta}$. This shows that $\fg^*_{-r,\vartheta}$ is stable. Combining all the above arguments, we have the following proposition 
\begin{prop}
    For $r \in \Z_{(p)}\cap \Q_{>0}$ and $\vartheta \in \DL_r$, the subset $\fg^*_{-r,\vartheta}\subset \fg^*_{-r} \subset \fg^*$ is open, closed and stable. 
\end{prop}

Let $\mc H(\fg)$ denote the Hecke algebra of compactly supported smooth measures on $\fg$ and $C^\infty_c(\fg)$ denote compactly supported smooth functions on $\fg$. We denote by $D(\fg)$ the set of distributions on $\fg$, which is the linear dual of $C^\infty_c(\fg)$ and $\widehat C(\fg)$ the set of generalized functions, which is the linear dual of $\mc H(\fg)$. Similar definitions can be made for $\fg^*$. Further let $D^G(\fg)$ denote the space of invariant distributions and $\widehat C^G(\fg)$ denote the space of invariant generalized functions (similarly for $\fg^*$). Following \cite[Sec 1.18]{bkv2016depthrprojector}, the fixed additive character $\psi:F \rightarrow \C^\times $ gives rise to a Fourier transform $\mc F: \mc H(\fg^*) \rightarrow C^\infty_c(\fg)$  given explicitly by $\mc F(h)(a)=\int_{\fg^*}\psi(\langle-,a\rangle)h$ for $h \in \mc H(\fg^*)$ and $a \in \fg$. By duality, this gives a Fourier transform $\mc F:D(\fg)\rightarrow \widehat C(\fg^*)$, which restricts to $\mc F:D^G(\fg)\rightarrow \widehat C^G(\fg^*)$ explicitly given by $\mc F(E)(h)=E(\mc F(h))$ for $E \in D^G(\fg)$ and $h \in \mc H(\fg^*)$.

\paragraph{Lie algebra analogue of the projector} We define the Lie algebra analogue of the center $Z_\fg\subset D^G(\fg)$ following \cite{bkv2016depthrprojector} to be the subspace of $E \in D^G(\fg)$ such that $E*h$ has compact support for all $h \in \mc H(\fg)$ or equivalently $\mc F(E)\in C^G(\fg^*)$ is locally constant. Define $\mc E_{r, \vartheta}= \mc F^{-1}(1_{\fg^*_{-r,\vartheta}})$, and $\mc E_{r, \vartheta} \in Z_{\fg}$ since we have shown that $\fg^*_{-r,\vartheta}$ is open and closed, and hence $1_{\fg^*_{-r,\vartheta}}$ is locally constant. 

\paragraph{Quasi-logarithms} Assume $G$ is semi-simple simply connected and $p$ is very good for $G$ in the sense of \cite[8.10]{bkv2016depthrprojector}. From \cite[Appendix C]{bkv2016depthrprojector}, we have an $\Ad\:G$-equivariant quasi-logarithm $\mcL$ defined over $\mf O$ which induces analytic isomorphisms $\mc L_r:G_r \rightarrow \fg_r$ and $\mc L_{x,r}: G_{x,r}\rightarrow \fg_{x,r}$ for $r >0$ and $x \in \B(G,F)$. Further, this logarithm also induces the Moy-Prasad isomorphisms 
\begin{equation*}
    \bar{\mc L}_{x,r}: \frac{G_{x,r}}{G_{x,r+}} \xrightarrow{\sim} \frac{\fg_{x,r}}{\fg_{x,r+}}
\end{equation*}
Hence, the following diagram commutes for $\sigma\in [\B_m]$
 \begin{equation*}
     \begin{tikzcd}
         G_{\sigma,r} \arrow{r}{\mc L_{\sigma,r}}\arrow{d}&\fg_{\sigma,r}\arrow{d} \\
        \frac{G_{\sigma,r}}{G_{\sigma,r+}}\arrow{r}{\bar{\mc L}_{\sigma,r}} & \frac{\fg_{\sigma,r}}{\fg_{\sigma,r+}}
     \end{tikzcd}
 \end{equation*}

\begin{rmk}
   Note that the $r$-logarithms $\mc L_r$ and $\mc L_{x,r}$ are analytic isomorphisms, not group theoretic. However, on Moy-Prasad quotients, it induces an isomorphism of vector spaces which is $\G_{x}$-equivariant since the $\Ad (G)$-equivariant quasi-logarithm $\mc L:G \rightarrow \fg $ in \cite{bkv2016depthrprojector} is defined over $\mf O$, has $\mc L(1)=0$ and the induced map on tangent spaces is identity. Since our group is split, the Moy-Prasad isomorphisms hold and the higher terms in the analytic expression vanishes in the quotient which gives compatibility with group operations without any additional restriction on $p$. 
\end{rmk}
Following \cite[Section 5]{bhattacharya2025description}, we have a natural inclusion map $\iota_{\sigma,r}: C(\bg_{\sigma,r}) \rightarrow C_c^\infty\left(\frac{G(F)/G_{\sigma,r+}}{G_{\sigma,0}}\right)=\M^r_\sigma$ sending $ C(\bg_{\sigma,r}) \ni f: \bg_{\sigma,r} \rightarrow\C$ to $\iota_{\sigma,r}( f) \in \M^r_\sigma$ supported in $G_{\sigma,r} \subset G(F)$ given by 
\begin{equation*}
    \iota_{\sigma,r}( f) :G_{\sigma,r}\rightarrow         \frac{G_{\sigma,r}}{G_{\sigma,r+}} \xrightarrow[\simeq]{\bar{\mc L}_{\sigma,r}} \frac{\fg_{\sigma,r}}{\fg_{\sigma,r+}} =\bg_{\sigma,r} \xrightarrow{f}\C
\end{equation*}

We can define another natural inclusion map $\tilde \iota_{\sigma,r}:C(\bg_{\sigma,r}) \rightarrow C^\infty_c(\fg) $ sending $ C(\bg_{\sigma,r}) \ni f: \bg_{\sigma,r} \rightarrow\C$ to  $\iota_{\sigma,r}( f)\in C^\infty_c(\fg)$ supported in $\fg_{\sigma,r} \subset \fg$ given by 
\begin{equation*}
    \tilde \iota_{\sigma,r}(f): \fg_{\sigma,r} \rightarrow  \frac{\fg_{\sigma,r}}{\fg_{\sigma,r+}} =\bg_{\sigma,r} \xrightarrow{f}\C
\end{equation*}
 Observe that $\tilde\iota_{\sigma,r}( f) = (\mc L_{\sigma,r})_!(\iota_{\sigma,r}( f))$. 
Fix a Haar measure $\mu$ on $G(F)$ similarly to \cite{bhattacharya2025description}, and let $\tilde \mu:=\mc L_!(\mu)$ be a Haar measure on $\fg$. Fixing the Haar measure $\tilde \mu$ fixes an isomorphism $\mc H(\fg) \xrightarrow{\sim} C_c^\infty(\fg)$. We define $c^\mu_{\sigma,r} $ to be the constant $\mu(G_{\sigma,r+})\left|\bg_{\sigma,r}\right|^{1/2}= \tilde \mu (\fg_{\sigma,r+})\left|\bg_{\sigma,r}\right|^{1/2}$. For $\sigma \in [\bar\mcC_m]$, consider the composed map 
\begin{equation*}
    j_{\sigma,r}: \C[\bt^*//W(\F_q)] \xlongrightarrow{(c^\mu_{\sigma,r})^{-1}t^*_{\sigma}} \C[\bt^*//W_\sigma(\F_q)]\xlongrightarrow{\rho_{\sigma,r}} C^{\mathrm{st}}(\bg_{\sigma,r}) \xlongrightarrow{\iota_{\sigma,r}}\M^r_\sigma
\end{equation*}
 where $C^{\mathrm{st}}(\bg_{\sigma,r})$ and $\rho_{\sigma,r}$ are as defined in \cite[Section 4]{bhattacharya2025description}. Further we also define $\tilde  j_{\sigma,r} $ as the composed map 
 \begin{equation*}
       \tilde j_{\sigma,r}: \C[\bt^*//W(\F_q)] \xlongrightarrow{(c^\mu_{\sigma,r})^{-1}t^*_{\sigma}} \C[\bt^*//W_\sigma(\F_q)]\xlongrightarrow{\rho_{\sigma,r}} C^{\mathrm{st}}(\bg_{\sigma,r}) \xlongrightarrow{\tilde \iota_{\sigma,r}} C^\infty_c(\fg)
 \end{equation*}
 and since $j_{\sigma,r}$ has image supported in $G_{\sigma,r}$, we have $\tilde j_{\sigma,r}=(\mc L_{\sigma,r})_!\circ j_{\sigma,r}$. Let $i_{\sigma,r}:=\rho_{\sigma,r}\circ (c^\mu_{\sigma,r})^{-1}t^*_{\sigma}$ as defined in \cite[Lemma 5.3]{bhattacharya2025description}.

 For $\sigma\in [\bar \mcC_m]$, we define $E_{\sigma,r,\vartheta}:=j_{\sigma,r}(1_\vartheta)$ and $\mc E_{\sigma,r,\vartheta}:= \tilde j_{\sigma,r}(1_\vartheta) $, and extend the definition to $\sigma \in [\B_m]$ by the adjoint action. For each non-empty finite $\Sigma \in \Theta_m$, define $E^\Sigma_{r,\vartheta}:= \sum_{\sigma \in \Sigma}(-1)^{\mathrm{dim}(\sigma)}E_{\sigma,r,\vartheta}d\mu \in \mc H(G)$, and let $E_{r,\vartheta}=\lim_{\Sigma \in \Theta_m}E^\Sigma_{r,\vartheta}$. From the proof of \cite[Theorem 1.1]{bhattacharya2025description}, we know that $\xi^r(1_\vartheta)=E_{r,\vartheta} $ interpreted as $G(F)$-invariant essentially compact distributions. 
 
 We make similar definitions for the Lie algebra and the dual of the Lie algebra. For $\Sigma \in \Theta_m$, define $\mc E^\Sigma_{r,\vartheta}:= \sum_{\sigma \in \Sigma}(-1)^{dim(\sigma)}\mc E_{\sigma,r,\vartheta}d\tilde \mu \in \mc H(\fg)$. In the dual of the Lie algebra, set $\fg^*_{\Sigma,-r,\vartheta}:=\cup_{\sigma \in \Sigma}\fg^*_{\sigma ,-r,\vartheta}$. We have the following lemma, which will be used in the proof of the next proposition 
\begin{lemma}\label{lemma:stabilityaltsum}
    For every $\Sigma \in \Theta_m$ and $r \in \frac{1}{m}\Z$, we have the equality $  1_{\fg^*_{\Sigma,-r,\vartheta}} = \sum_{\sigma \in \Sigma }(-1)^{\dim(\sigma)}1_{ \fg^*_{\sigma ,-r,\vartheta}}$
\end{lemma}
 \begin{proof}
     By \cite[Lemma 6.8]{bkv2016depthrprojector}, we have
$$1_{\mathfrak{g}^*_{\Sigma,-r}}=\sum_{\sigma\in\Sigma}(-1)^{\mathrm{dim }(\sigma)}1_{\mathfrak{g}^*_{\sigma,-r}}.$$ Multiplying both sides by $1_{\mathfrak{g}^*_{-r,\vartheta}}$ gives us the equation we want.
 \end{proof}
  
The next proposition proves that the Fourier inverse of the characteristic function on the ``dual blob" is exactly the Lie algebra analogue of our projector.
 \begin{prop}\label{prop:liealgebraprojlim}
     For any $f \in C^\infty_c(\fg)$, the inductive system $\{\mc E^\Sigma_{r,\vartheta}(f)\}_{\Sigma \in \Theta_m}$ stabilizes, and we have $\mc E_{r,\vartheta}(f)=\lim_{\Sigma \in \Theta_m}\mc E^\Sigma_{r,\vartheta}(f)$. In particular, $\mc E_{r,\vartheta}$ is supported in $\fg_r$. 
 \end{prop}
\begin{proof}
    The restriction of the Fourier transform $\mc F:D(\fg)\rightarrow \widehat C(\fg^*)$ to $\mc H(\fg)$ is the Fourier transform $\mc F: \mc H(\fg) \rightarrow C^\infty_c(\fg^*)$ (which we again denote by $\mc F$).  Fixing the Haar measure, we can write $ \mc F:C_c^\infty(\fg) \rightarrow C_c^\infty(\fg^*)$ explicitly as 
\begin{equation*}
    F(h)(a)=\int_\fg \psi(\langle a,y\rangle)h(y)d\tilde \mu(y)
\end{equation*}
For every $\Sigma \in \Theta_m$, we had defined $\fg^*_{\Sigma,-r,\vartheta}:=\cup_{\sigma \in \Sigma}\fg^*_{\sigma ,-r,\vartheta}$ which is an open and compact subset, and $\fg^*_{-r,\vartheta}=\cup_{\Sigma \in \Theta_m}\fg^*_{\Sigma,-r,\vartheta}$. Thus $1_{\fg^*_{-r,\vartheta}} = \lim_{\Sigma \in \Theta_m}1_{\fg^*_{\Sigma,-r,\vartheta}}$ and 
\begin{equation*}
    \mc E_{r,\vartheta} = \mc F^{-1}(1_{\fg^*_{-r,\vartheta}})=  \lim_{\Sigma \in \Theta_m}\mc F^{-1}(1_{\fg^*_{\Sigma,-r,\vartheta}})
\end{equation*}
It therefore suffices to show that $\mc F^{-1}(1_{\fg^*_{\Sigma,-r,\vartheta}})=\mc E_{r,\vartheta}^\Sigma$, i.e., $\mc F(\mc E_{r,\vartheta}^\Sigma)=1_{\fg^*_{\Sigma,-r,\vartheta}}$. Lemma \ref{lemma:stabilityaltsum} shows that $1_{\fg^*_{\Sigma,-r,\vartheta}} = \sum_{\sigma \in \Sigma }(-1)^{dim(\sigma)}1_{ \fg^*_{\sigma ,-r,\vartheta}}$ and hence by definition of $\mc E_{r,\vartheta}^\Sigma$, it is enough to show that $\mc F(\mc E_{\sigma,r,\vartheta})= 1_{\fg^*_{\sigma,-r,\vartheta}}$. Using the pairing $\fg^* \times \fg\rightarrow \C^\times $, properties of Fourier transforms and the fact that the function is invariant under right translation by $\mathfrak{g}_{\sigma,r+}$ we know that the support of $\mc F(\mc E_{\sigma,r,\vartheta})$ is contained in $\fg^*_{\sigma,-r}$.  For $a \in \fg^*_{\sigma,-r}$ and $y \in \fg_{\sigma,r}$, let $\bar a$ and $\bar y$ denote the images in $\frac{\fg^*_{\sigma,-r}}{\fg^*_{-r+}}$ and $\frac{\fg_{\sigma,r}}{\fg_{\sigma,r+}}$ respectively. Further, for $r >0$, the pairing $\langle-,-\rangle$ induces a $\G_\sigma(\F_q)$-invariant pairing 
\begin{align*}
  \frac{\fg^*_{\sigma,-r}}{\fg^*_{\sigma,-r+}} \times   \frac{\fg_{\sigma,r}}{\fg_{\sigma,r+}} &\rightarrow \F_q\\
  (X,Y) & \mapsto \langle X,Y\rangle \:\mathrm{mod}\:\varpi \mf O
\end{align*}
of $\F_q$-vector spaces. Using the above facts, we have for $a \in \fg^*_{\sigma ,-r}$
\begin{align*}
    \mc F(\mc E_{\sigma,r,\vartheta})(a)&= \int_\fg \psi(\langle a,y\rangle)\mc E_{\sigma,r,\vartheta}(y)d\tilde \mu(y)= \int_{\fg_{\sigma,r}} \psi(\langle a,y\rangle)\mc E_{\sigma,r,\vartheta}(y)d\tilde \mu(y)\\
    &= \tilde \mu (\bg_{\sigma,r+}) \sum_{\bar y \in \bg_{\sigma,r}}\tilde \psi((\bar a,\bar y))i_{\sigma,r}(1_\vartheta)(\bar y)=c^\mu_{\sigma,r}\mc F_{\bg_{\sigma,r}}(i_{\sigma,r}(1_\vartheta))(\bar a)\\
    & =\mc F_{\bg_{\sigma,r}}(\rho_{\sigma,r}\circ t^*_\sigma(1_\vartheta))(\bar a)= \mc F_{\bg_{\sigma,r}}\circ \mc F_{\bg^*_{\sigma,r}}(\chi_{\sigma,r}^*(t^*_\sigma(1_\vartheta))^-)(\bar a)
\end{align*}
where the $\mc F_{\bg_{\sigma,r}}$ and $\mc F_{\bg^*_{\sigma,r}}$ are as defined in \cite[Section 4.2]{bhattacharya2025description}. Thus we have for $a \in \fg^*$ 
\begin{equation*}
    \mc F(\mc E_{\sigma,r,\vartheta})(a)= \chi_{\sigma,r}^*(t^*_\sigma(1_\vartheta))(\bar a) = 1_{J_{\sigma,r}^{-1}(\vartheta)}(a)=1_{\fg^*_{\sigma,-r,\vartheta}}(a)
\end{equation*}
finishing our proof.

\end{proof}

Note that we have a quasi-logarithm $\mc L$ which for $r>0$ induces $r$-logarithms (similar to \cite[1.21]{bkv2016depthrprojector}) $\mc L_r: G_r \rightarrow \fg_r$. Note that we use $G_r$ and $G_{x,r}$ and $r>0$ for our convention of $r$-logarithm, instead of $G_{r+}$ as in \cite{bkv2016depthrprojector}. Using that, we have the following lemma, which follows directly from the construction of $j_{\sigma,r}$ and $\tilde j_{\sigma,r}$
\begin{lemma}\label{lemma:pushforward}
    The pushforward $(\mc L_r)_!(E_{r,\vartheta}|_{G_r})$ equals $\mc E_{r,\vartheta}|_{\fg_r}$.
\end{lemma}
\begin{proof}
    The $r$-logarithm $\mc L_r$ induces homeomorphisms $\mc L_{\sigma,r}:G_{\sigma,r}\xrightarrow{\sim} \fg_{\sigma,r}$ for every $\sigma \in [\B_m]$, and we have $\tilde j_{\sigma,r}=(\mc L_{\sigma,r})_!\circ j_{\sigma,r}$. Hence, we have $(\mc L_r)_!(E^\Sigma_{r,\vartheta}|_{G_r})= \mc E_{r,\vartheta}^\Sigma|{\fg_r}$ for every $\Sigma \in \Theta_m$, and hence we can conclude by \cite{bhattacharya2025description} and Proposition \ref{prop:liealgebraprojlim}.

\end{proof}

\paragraph{Stability} We roughly follow the ideas in \cite[Section 8]{bkv2016depthrprojector}. By \cite[8.3(d), 8.4(a)]{bkv2016depthrprojector}, we know that $1_{\fg^*_{-r,\vartheta}} \in \widehat C^G(\fg^*)$ is stable. By a generalization in \cite{kazhdan2003generalization} of a theorem of Waldspurger \cite{waldspurger1997lemme}, we know that the Fourier transform preserves stability and hence $\mc E_{r,\vartheta}= \mc F^{-1}(1_{\fg^*_{-r,\vartheta}})$ is stable. 

In \cite[Lemma B.3, 8.4(a)]{bkv2016depthrprojector}, it is shown that $1_{G_r}$ is stable. Hence, for any unstable function $f \in C^\infty_c(G)$, we know that the functions $f':=f\cdot1_{G_r}$ and $(\mc L_r)_!(f')$ are unstable. Further, since $E_{r,\vartheta}$ is supported on $G_r$, we have $E_{r,\vartheta}(f)=E_{r,\vartheta}(f')$ and $E_{r,\vartheta}(f') = \mc E_{r,\vartheta}((\mc L_r)_!(f'))$ by Lemma \ref{lemma:pushforward}. Hence, $E_{r,\vartheta}(f)=\mc E_{r,\vartheta}((\mc L_r)_!(f'))=0$ since $\mc E_{r,\vartheta}$ is stable and $(\mc L_r)_!(f')$ is unstable.

\paragraph{}Note that this proves our result for semi-simple simply connected groups $G$ when $p$ is very good for $G$. We extend to more general groups using the following results
\begin{lemma}[\cite{bkv2016depthrprojector}, Corollary 8.11]\label{lemma:verygoodlog}
    If $p$ is very good for $G$, the simply connected cover of the derived subgroup $G^{sc}$ admits an $r$-logarithm for $r \in \R_{>0}$. 
\end{lemma}

\begin{lemma}[\cite{bkv2016depthrprojector}, Lemma 8.12]\label{lemma:isogenyisom}
    Let $\pi:G'\rightarrow G$ be an isogeny of degree prime to $p$. Then $\pi$ induces homeomorphisms $G'_{x,r} \xrightarrow{\sim} G_{x,r}$ and $G'_r\xrightarrow{\sim}G_r$ for all $r \in \R_{>0}$ and $x \in \B(G,F)$
\end{lemma}
For a central isogeny $\pi:G'\rightarrow G$, we have an isogeny of maximal tori $T'\rightarrow T$ and hence there is a natural identification $\bt'^*//W \xrightarrow{\sim}\bt^*//W$. 
\begin{lemma}
   Let $\pi:G'\rightarrow G$ be a central isogeny of degree prime to $p$. The distribution $E_{r,\vartheta}$ on $G(F)$ is stable if and only if $E_{r,\vartheta}$ on $G'(F)$ is stable. 
\end{lemma}
\begin{proof}
    Since $E_{r,\vartheta}$ is supported on $G_r$, it is enough to check whether $E_{r,\vartheta}(f)=0$ for every unstable $f$ supported on $G_r$. This follows from Lemma \ref{lemma:isogenyisom} and the fact that $E_{r,\vartheta} = \lim_{\Sigma \in \Theta_m}E^\Sigma_{r,\vartheta}$. 
\end{proof}
Combining the above results, we have our theorem
\begin{thm}
    Let $p$ be very good for $G$ in the sense of \cite[8.10]{bkv2016depthrprojector}, $r \in \Z_{(p)}\cap \Q_{> 0}$ and $\vartheta \in \DL_r=\bt^*//W(\F_q) $. Then, the invariant distribution $E_{r,\vartheta}$ is stable.
\end{thm}
\begin{proof}
     Consider the central isogeny $\pi:G^{sc} \times Z(G)^0 \rightarrow G$. Following \cite[8.14]{bkv2016depthrprojector}, we see that the degree of $\pi$ is prime to $p$. Hence, by using Lemma \ref{lemma:isogenyisom}, it is enough to show the stability of $E_{r,\vartheta}$ on $G^{sc}\times Z(G)^0$. By Lemma \ref{lemma:verygoodlog}, $G^{sc} $ admits an $r$-logarithm and since $T$ is $F$-split, $Z(G)^0$ is split and admits an $r$-logarithm. Hence our assertion follows from previous results.
\end{proof}

\printbibliography

\bigskip
\textsc{Department of Mathematics, University of Minnesota, Minneapolis, MN-55455}\\
\textit{E-mail address}: \texttt{bhatt356@umn.edu} 

\medskip

\textsc{Department of Mathematics, University of Maryland, College Park, MD-20742 }\\
\textit{E-mail address}: \texttt{shenghao@umd.edu}
\end{document}